\documentclass[12pt]{article}
\usepackage[top=1in, bottom=1in, left=1in, right=1in]{geometry}

\usepackage{hyperref}
\hypersetup{
    colorlinks = true,
    linkcolor = red,
    anchorcolor = red,
    citecolor = red,
    filecolor = blue,
    urlcolor = blue
    }
    
\usepackage{multirow}
\usepackage{cite}
\usepackage{caption}

\usepackage{amsmath}
\usepackage{amsfonts}
\usepackage{amssymb}
\usepackage{amsthm}

\usepackage{algorithmic}

\global\long\def\trace{\operatorname{trace}}

\newtheorem{thm}{Theorem}

\newtheorem{prop}{Proposition}

\newtheorem{cor}{Corollary}

\renewcommand{\AA}{\mathcal{A}}
\newcommand{\CC}{\mathcal{C}}
\newcommand{\DD}{\mathcal{D}}

\newcommand{\NN}{\mathcal{N}}
\newcommand{\E}{\mathbb{E}}
\newcommand{\N}{\mathbb{N}}
\renewcommand{\P}{\mathbb{P}}
\newcommand{\Q}{\mathbb{Q}}
\newcommand{\R}{\mathbb{R}}

\newcommand{\one}{\textbf{1}}   

\newcommand{\AS}{\textbf{AS}}

\DeclareMathOperator{\cpl}{\mathfrak{C}}
\DeclareMathOperator{\KL}{KL}

\newcommand{\mkp}{\left\lfloor \frac{mk}{p} \right\rfloor}

\usepackage{graphicx}
\usepackage{tikz}
\usetikzlibrary{calc}

\xdefinecolor{mygray}{rgb}{ 0.4275, 0.4549,   0.4471}

\begin{document}
\title{Detection of Correlations with Adaptive Sensing}
\author{
Rui M. Castro,
G\'{a}bor Lugosi,
Pierre-Andr\'{e} Savalle%
\thanks{R. Castro is with the Department of Mathematics, Eindhoven University of Technology, 5600 MB Eindhoven,
The Netherlands (email: \texttt{rmcastro@tue.nl});
G. Lugosi is with ICREA and the Department of Economics, Pompeu Fabra University, Ramon Trias Fargas 25-27, 08005 Barcelona, Spain (email: \texttt{gabor.lugosi@upf.edu}). His work was supported by the Spanish Ministry of Science and Technology grant MTM2012-37195;
P.A. Savalle is with CMLA, ENS Cachan, UMR 8536 CNRS, and Laboratoire MAS, Ecole Centrale Paris, Grande Voie des Vignes, 92290 Chatenay-Malabry, France (email: \texttt{pierre-andre.savalle@ecp.fr}).
}%
}

\maketitle

\begin{abstract}
The problem of detecting correlations from samples of a
high-dimensional Gaussian vector has recently received a lot of
attention.  In most existing work, detection procedures are provided
with a full sample.  However, following common wisdom in experimental
design, the experimenter may have the capacity to make targeted
measurements in an on-line and adaptive manner.  
In this work, we investigate such adaptive sensing procedures for detecting 
positive correlations. It it shown that, using the same number of 
measurements, adaptive procedures are able to detect significantly 
weaker correlations than their non-adaptive counterparts.
We also establish minimax lower bounds that show the limitations of
any procedure.  
\end{abstract}

\maketitle
\section{Introduction}
{
In this paper we consider a statistical testing problem related to anomaly detection: the detection of correlations between signals.
In the general problem of anomaly detection, one aims to identify
unexpected activity in data. It has applications in numerous domains \cite{chandola2009anomaly}, such as finance \cite{barber1997detecting}, computer security \cite{hofmeyr1998intrusion}, health monitoring \cite{lin2005approximations}, or detection of activity in sensor networks \cite{noble2003graph,janakiram2006outlier,van2006anomaly}.
In many situations, anomalies can be detected by looking at unusual signal values at any of the sensors. For instance, a home security alarm is usually comprised of various infrared or related sensors, and an alert is raised as soon as a single sensor detects an unusual signal.
However, in other situations, when signals are ``weak'', they may never appear anomalous in isolation, and anomalies may only be detected when considering the signals together as a collection.
This type of phenomena may be referred to as either
contextual anomaly detection\cite{song2007conditional}, 
or collective anomaly detection \cite{shekhar2001detecting}, depending on the setup.
A prototypical example of such a problem is the detection of Distributed Denial-of-Service (DDoS) attacks in computer networks, which has become an important challenge  in recent years \cite{wang2002detecting,thottan2003anomaly,moore2006inferring}.  In a DDoS attack, the attacker usually controls a large number of computers  distributed around the world. These machines are used to simultaneously send requests to a target server, which is then flooded by the amount of packets, and can become unavailable as a result. 
As a side effect, this type of attack can produce high volumes of traffic in various parts of the worldwide internet infrastructure. 
However, packets sent by the attacker through the machines that he/she controls cannot usually be detected as anomalous in isolation \cite{jung2002flash}, and detection of DDoS requires to correlate signals obtained at different points in the network. 
Collective anomalies also appear, for instance, in the context of detection of the outbreak of diseases \cite{kulldorff2005space}. 
Another important type of anomaly detection problem appears when dealing with sensor data arranged on a two-dimensional grid (e.g., loop detectors in lanes of road networks, or wireless sensor networks \cite{akyildiz2002wireless}). In this case, collective anomalies may be characterised by {\it neighbouring} signals being correlated.
Besides anomaly detection, detection of correlations is also of interest to assess to what extent dimensionality reduction can be performed on a data stream.  Reduction of dimensionality is a workhorse of data analysis, and there has been a strong recent interest in modifying principal component analysis to deal with high-dimensional data \cite{johnstone2009consistency, berthet2012optimal, cai2013optimal}. Testing when this type of transformation is justified is thus an important problem.

In this work, we consider a simple correlation model: given multiple observations from a Gaussian multivariate distribution we want to test whether the corresponding covariance matrix is diagonal against non-diagonal alternatives. 
}
Such problems have recently received a lot of attention in the
literature, where different models and choices of non-diagonal
covariance alternatives were considered
\cite{hero2012hub,arias2012detecting, arias2012detection,
 berthet2012optimal, cai2013optimal}. We consider the
detection of sparse {\it positive} correlations, which has been
treated in the case of a unique multivariate sample
\cite{arias2012detecting}, or of multiple samples
\cite{arias2012detection}. However, this paper deviates from the
existing literature in that we consider an \emph{adaptive sensing} or
\emph{sequential experimental design} setting. More precisely, data is
collected in a sequential and adaptive way, where data collected at
earlier stages informs the collection of data in future
stages. Adaptive sensing has been studied in the context of other
detection and estimation problems, such as in detection of a shift in
the mean of a Gaussian vector \cite{castro2012adaptive,
  haupt2009distilled}, in compressed sensing
\cite{arias2011fundamental, haupt2012SCS,castro2012adaptive}, in experimental design,
optimization with Gaussian processes \cite{srinivas2009Gaussian}, and
in active learning \cite{chennear}. Adaptive sensing procedures are
quite flexible, as the data collection procedure can be
``steered'' to ensure most collected data provides important
information. As a consequence, procedures based on adaptive sensing
are often associated with better detection or estimation performances
than those based on non-adaptive sensing with a similar measurement
budget. 
In this paper, our objective is to determine whether this is
also the case for detection of sparse positive correlations, and if
so, to quantify how much can be gained.

\subsection{Model}\label{sec:model}

Let $U^t\in\R^n$, $t=1,2,\ldots$ be independent and identically distributed (i.i.d.) normal random vectors with zero mean and covariance matrix $\Sigma_S$, where $S$ is a subset of $[n] = \{1,\ldots,n\}$. Let $\rho>0$ and define the covariance matrix as
\[(\Sigma_S)_{i,j}=\left\{\begin{array}{ll}
1, & i=j\\
\rho, & i\neq j, \text{ with } i,j\in S\\
0, & \text{ otherwise}.
\end{array}
\right.
\]
Our main goal is to solve the hypothesis testing problem
\begin{align*}
H_0:\, & S = \emptyset\\
H_1:\, & S \in \CC,
\end{align*}
where $\CC$ is some class of non-empty subsets of $\{1, \ldots, n\}$,
each of size $k$. In other words, under the alternative hypothesis,
there exists an unknown subset $S \in \CC$ such that corresponding
components are positively correlated with strength $\rho > 0$. We 
often refer to the elements of $S$ as the subset of {\it contaminated}
coordinates. 
{
The model of correlations we consider appears naturally in the problem
of detecting a sparse signal embedded in noise. Indeed, with
$(Y^t_i)$ and $N^t$ being independent standard normal random variables, and
\[
U^t_i =
\left\{ \begin{array}{ll}
Y^t_i,\, & i \notin S,\\
\sqrt{1-\rho} Y^t_i + \sqrt{\rho} N^t ,\, & i \in S
\end{array} \right. 
\]
for some $S \in \CC$, then the vectors $U^t$ are independent
multivariate zero-mean normal vectors with covariance matrix
$\Sigma_S$. The variable $N_t$ represent a common signal present at each 
contaminated coordinate and $Y^t_i$ the additive white noise.
}
In all cases we assume that the cardinality of each $S\in \CC$ is the same:
$|S|=k$.
We consider the following types of classes $\CC$ for the
contaminated coordinates:
\begin{itemize}
\item
{\bf  $k$-intervals}: all sets of $k$ contiguous coordinates, of the form $\{z, z+1, \ldots, z+k-1\}$ for some $1 \leq z \leq  n-k+1$; this class has size linear in $n$, and
we denote it by $\CC_{[k]}$.
\item
{\bf disjoint $k$-intervals}: the class $\DD_{[k]}$ defined as
 \[
\DD_{[k]} = \{I_1, \ldots, I_{\lfloor n/k \rfloor}\}
\]
where $I_j = \{(j - 1) k + 1, \ldots, j k \}, \, j \in \left\{1, \ldots, \lfloor n/k\rfloor \right\}.$
 \item
 {\bf $k$-sets}: all subsets of $\{1, \ldots, n\}$ of cardinality $k$. 
We denote this class by $\CC_k$.
\end{itemize}
{
In addition, it is of interest for applications to consider settings where the coordinates $\{1, \ldots, n\}$ are laid out according to a two-dimensional grid  $[ n_1 ] \times [ n_2 ]$ with $n_1 n_2 = n$, similarly to a spatially arranged array of sensors.
Although $k$-sets still make sense in this setting, the contaminated set can be further assumed in this case to be connected and spatially localized in some sense. The following example is most intuitive:
\begin{itemize}
\item {\bf $(k_1, k_2)$-rectangles}: for $k_1 k _2 = k$, this comprises all sets of the form
\[
\{i_0, \ldots, i_0 + k_1 -1 \} \times \{j_0, \ldots, j_0 + k_2 -1 \}
\]
for $i_0 \in [n_1 - k_1 + 1]$, $j_0 \in [n_2 - k_2 + 1]$.
\end{itemize}
Results for rectangles or similar two-dimensional shapes can be
obtained easily from our results for $k$-intervals, and are identical
up to constants. We omit the rather straightforward details here.
}

For any $t=1,2,\ldots$ denote by $\P_\emptyset$ the distribution of
$U^t$ under the null, and by $\P_S$ the distribution under the
alternative with contaminated set $S \in \CC$. In addition, for a
positive integer $q$, we denote by $\P^{\, \otimes q}$ the product measure
$\P \otimes \ldots \otimes \P$ with $q$ factors. As previously, we let $[q] =
\{1, \ldots, q\}$.

\subsection{Adaptive vs. Non-Adaptive Sensing and Testing}\label{sec:sensingmodel}
Clearly, the above hypothesis testing problem would be trivial if one has access to an infinite number of i.i.d. samples $(U^t)_{t\in\{1,\ldots,\infty\}}$. Therefore, one must include some further restrictions on the data that is made available for testing. In particular, we only consider testing procedures that make use of at most $M$ entries of the matrix $(U_i^t)_{t\in\{1,\ldots,\infty\},i\in[n]}$. It is useful to regard this as a matrix with $n$ columns and an infinite number of rows.

The key idea of adaptive sensing is that information gleaned from previous observations can be used to guide the collection of future observations. To formalize this idea consider the following notation: for any subset $A \subseteq [n]$ we denote by $|A|$ the cardinality of $A$. When $A$ is nonempty we write $U_A = (U_i)_{i\in A} \in \R^{|A|}$ for the subvector of a vector $U \in\R^n$ indexed by coordinates in $A$. Finally, if $U$ is a random variable taking values in $\R^n$ denote by $\P|_A$ the distribution of $U_A$.

Let $S \in \CC \cup \left\{\emptyset\right\}$ be the set of
contaminated coordinates, and $M \geq 2$ be an integer. In our model
we are allowed to collect information as follows. We consider
successive rounds. At round $t \in \N$, one chooses a non-empty {\it
  query} subset $A^t \subseteq [n]$ of the components, and observes
$U_{A^t}^t$. To avoid technical difficulties later on, we define the
observation made at time $t$ as $X^t$, so that $X^t_{A^t}=U_{A^t}^t$ and
$X^t_{[n]\setminus A^t}=\mathbf{0}$. In words, one observes the $A^t$
coordinates of $U^t$, while the remaining coordinates are completely
uninformative. Each successive round proceeds in the same fashion,
under the requirement that the budget constraint
\begin{equation}\label{eqn:budget}
\sum_{t=1}^\infty |A^t| \leq M
\end{equation}
is satisfied. Note that clearly, the number of rounds is not larger than
$M$. Again, to avoid technical difficulties we assume the total number of
rounds to be $M$ in what follows, even if this means $A^t=\emptyset$
for some values of $t$. 
See Figure \ref{fig:sensing_model} for an illustration.

\begin{figure}[h!]
\centering
\tikzset{
    old inner xsep/.estore in=\oldinnerxsep,
    old inner ysep/.estore in=\oldinnerysep,
    double circle/.style 2 args={
        circle,
        old inner xsep=\pgfkeysvalueof{/pgf/inner xsep},
        old inner ysep=\pgfkeysvalueof{/pgf/inner ysep},
        /pgf/inner xsep=\oldinnerxsep+#1,
        /pgf/inner ysep=\oldinnerysep+#1,
        alias=sourcenode,
        append after command={
        let     \p1 = (sourcenode.center),
                \p2 = (sourcenode.east),
                \n1 = {\x2-\x1-#1-0.5*\pgflinewidth}
        in
            node [inner sep=0pt, draw, circle, line width=0.8pt, minimum width=2*\n1,at=(\p1),#2] {}
        }
    },
    double circle/.default={2pt}{blue}
}

\begin{tikzpicture}[transform shape, scale = 0.5]
\foreach \number in {2,...,9}{
\foreach \numberd in {2,...,6}{
	\node[circle, inner sep=0.1cm, fill=mygray] (N-\number-\numberd) at (\number, \numberd) {};
}
}
  
\foreach \A in {5-3, 5-4, 6-3, 6-4, 6-5, 5-5} {
\foreach \B in {5-3, 5-4, 6-3, 6-4,  6-5, 5-5} {
	\draw[line width=1pt,red!50] (N-\A) -- (N-\B);
}
}

\foreach \number in {2,...,9}{
\foreach \numberd in {2,...,6}{
	\node[double circle={-2pt}{black},inner sep=0.1cm, fill=mygray] (N-\number-\numberd) at (\number, \numberd) {};
}
}

\foreach \pair in {5-3, 5-4, 6-3, 6-4, 6-5, 5-5} {
 	\node[circle, inner sep=0.1cm, fill=red] (R-\pair) at (N-\pair) {};
}
\end{tikzpicture}

\vspace{0.5cm}

\begin{tikzpicture}[transform shape, scale = 0.5]
\foreach \number in {2,...,9}{
\foreach \numberd in {2,...,6}{
	\node[circle, inner sep=0.1cm, fill=mygray] (N-\number-\numberd) at (\number, \numberd) {};
}
}
  
\foreach \A in {5-3, 5-4, 6-3, 6-4, 6-5, 5-5} {
\foreach \B in {5-3, 5-4, 6-3, 6-4, 6-5, 5-5} {
	\draw[line width=1pt,red!50] (N-\A) -- (N-\B);
}
}

\foreach \A in {4-5,3-2,3-3,5-3, 3-4, 4-2,4-3, 5-2,5-4, 6-2, 7-3, 8-3, 9-3, 6-6, 7-6, 7-5,6-3, 6-4, 7-4, 6-5, 5-5, 4-4} {
	\node[double circle={-2pt}{black},inner sep=0.1cm, fill=mygray] (C-\A) at (N-\A) {};
}

\foreach \pair in {5-3, 5-4, 6-3, 6-4, 6-5, 5-5} {
 	\node[circle, inner sep=0.1cm, fill=red] (R-\pair) at (N-\pair) {};
}
\end{tikzpicture}

\vspace{0.5cm}

\begin{tikzpicture}[transform shape, scale = 0.5]
\foreach \number in {2,...,9}{
\foreach \numberd in {2,...,6}{
	\node[circle, inner sep=0.1cm, fill=mygray] (N-\number-\numberd) at (\number, \numberd) {};
}
}
  
\foreach \A in {5-3, 5-4, 6-3, 6-4, 6-5, 5-5} {
\foreach \B in {5-3, 5-4, 6-3, 6-4, 6-5, 5-5} {
	\draw[line width=1pt,red!50] (N-\A) -- (N-\B);
}
}

\foreach \A in {4-5,3-2, 4-2,4-3, 5-4,7-3, 8-3, 9-3, 6-6, 7-6, 7-5,5-5, 4-4} {
	\node[double circle={-2pt}{black},inner sep=0.1cm, fill=mygray] (C-\A) at (N-\A) {};
}
\foreach \pair in {5-3, 5-4, 6-3, 6-4, 6-5, 5-5} {
 	\node[circle, inner sep=0.1cm, fill=red] (R-\pair) at (N-\pair) {};
}
\end{tikzpicture}
\caption{Adaptive sensing over a two dimensional grid of sensors. 
The figure illustrates how
information can be obtained within the sensing model for $n=40$ and $k=6$, under
the alternative hypothesis with $S$ being a $(2, 3)$-rectangle in a $8 \times 5$ grid.
The correlated coordinates form a clique in the graph of correlations, and this is shown through light edges. At every step, the experimenter selects coordinates to be sensed, and these are shown circled.  At the first step, the experimenter samples all the coordinates, while at the two subsequent steps, the experimenter reduced the amount of coordinates sampled. This corresponds to a total budget of 
$|A^1| + |A^2| + |A^3| = 40 + 22 + 13 = 75$ coordinate measurements.
}
\label{fig:sensing_model}
\end{figure}
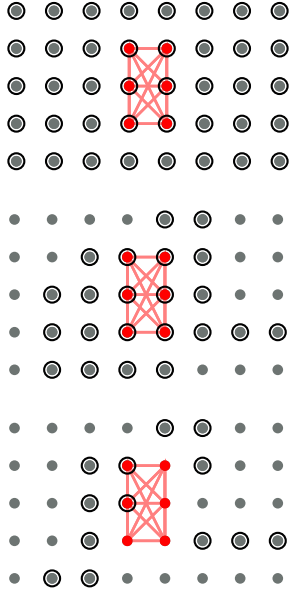

In our setting, one can select the query sequence randomly and sequentially, and hence, we write the query sequence $(a^1, \ldots, a^M)$ as a realization of a sequence $(A^1, \ldots, A^M)$ of $M$ random subsets of $[n]$, some of which may be empty, and such that $\sum_{t=1}^M |A^t| \leq M$.

A key aspect of adaptive sensing is that the query at round $T$ may depend on all the information available up to that point. We assume $A^t$ can depend on the history at time $t-1$, which we denote by $H^{t-1} = (A^j, X^j)_{j \in [t-1]}$. More precisely, we assume $A^t$ is a measurable function of  $H^{t-1}$, and possibly of additional randomization. We call the collection of all the conditional distributions of $A^t$ given $H^{t-1}$ for $t \in [M]$ the {\it sensing strategy}. In particular, if there is no additional randomization, $A^t$ is a deterministic function of $H^{t-1}$. We denote the set of all possible adaptive sensing strategies with sensing budget $M$ as $\AS(M)$.

At this point, it is important to formally %
clarify what is meant by
\emph{non-adaptive sensing}. This is simply the scenario where
$(A^t)_{t\in[M]}$ is independent of $(U_i^t)_{t\in[M]},i\in[n]$. In
other words, all the decisions regarding the collection of data must
be taken before any observations are made. The collection
$(A^t)_{t\in[M]}$ is known as a \emph{non-adaptive sensing strategy}.
A natural and important
choice is \emph{uniform sensing}, where $A^t=[n]$ for $t=1,\ldots,M/n$
(assume $M$ is divisible by $n$). In words, one collects $m=M/n$
i.i.d. samples from $\P_S$. This problem has been thoroughly studied
in \cite{arias2012detecting}; we summarize some of the main results of
\cite{arias2012detecting} in Section~\ref{sec:uniform_sensing}.

Now that we have formalized how data is collected, we can perform statistical tests. Formally, a {\it test} is a measurable binary function $\phi : H^M \mapsto \phi(H^M)\in\{0, 1\}$, that is, a binary function of all the information obtained by the (adaptive or non-adaptive) sensing strategy. The result of the test is $\phi(H^M)$, and if this is one we declare the rejection of the null hypothesis. Finally, an {\it adaptive testing procedure} is a pair $(\AA, \phi)$ where $\AA$ is a sensing strategy and $\phi$ is a test.

For any sensing strategy $\AA$ and $S \in\CC$, define $\P_\emptyset^\AA$ (resp. $\P_S^\AA$) as the distribution under the null (resp. under the alternative with contaminated set $S$) of the joint sequence $(A^1, X^1, \ldots, A^M, X^M)$ of queries and observations. The performance of an adaptive testing procedure $(\AA, \phi)$ is evaluated by comparing the worst-case risk
\[
R(\AA, \phi) =
\P_\emptyset^\AA(\phi \neq 0) + \max_{S \in \CC} \P_S^\AA(\phi \neq 1)
\]
to the corresponding minimax risk $R^*_\AS = \inf_{\AA\in\AS(M), \phi}
R(\AA, \phi)$, where the infimum is over all adaptive testing
procedures $(\AA, \phi)$ with a budget of $M$ coordinate
measurements. The minimax risk $R^*_\AS$ depends on $M$, although we
do not write this dependence explicitly for notational ease. 

Let $m = M/n$ be the equivalent number of {\it full vector measurements}. In the following, we will just say $m$ {\it measurements} for simplicity.
This change of parameters allows for easier comparison with the special case of uniform sensing, where a full vector of length $n$ is measured $m$ times.
In particular, when $m = M / n $ is an integer, uniform sensing corresponds to  the deterministic sensing procedure with $A^t = [n]$ for $t \in [m]$,  $A^t = \emptyset$ for $t > m$, and $\P_S^\AA = \P_S^{\,\otimes m}$ for $S \in \CC \cup \left\{\emptyset\right\}$.

We are interested in the {\it high-dimensional} setting, where the ambient dimension $n$ is high.
All quantities such as the correlation coefficient $\rho$, the contaminated set size $k$, and the number of vector measurements $m$ will thus be allowed to depend on $n$. In particular, we always assume that $n$, $k$ and $m$ all go to infinity simultaneously, albeit possibly at different rates,
and 
our main concern is to identify the range of parameters in which
it is possible to construct adaptive tests whose risks converge to zero.
We consider the sparse regime where $k = o(n)$.
 Although the case of fixed $\rho$ is of interest, most of our results will be concerned with the case where $\rho$ converges to zero with $n$.
When $\rho = 1$, the problem is trivial as detecting duplicate entries in a single sample vector from the distribution allows one to perform detection perfectly, while for fixed $\rho < 1$, the problem essentially becomes easier as the measurement budget $m$ increases.

\subsection{Uniform Sensing and Testing}\label{sec:uniform_sensing}

The simplest and most-natural type of non-adaptive sensing strategy we
can consider is uniform sensing. As stated before, this corresponds to
the choice $A^t=[n]$ for $t=1,\ldots,m$ (recall that $m=M/n$), that is
one collects $m$ i.i.d. samples from $\P_S$. The minimax risk and the
performance of several uniform sensing testing procedures have been
analyzed in \cite{arias2012detecting}. The authors of that work
analyzed the performance of tests based on the {\it localized squared
  sum} statistic
\[
T_{\text{loc}} = \max_{S \in \CC} \sum_{t=1}^m \left( \sum_{i \in S} X_i^t\right)^2,
\] 
which was shown to be near-optimal in a variety of scenarios. The
localized squared sum test that rejects the null hypothesis 
when $T_{\text{loc}}$ exceeds
a properly chosen threshold was shown to have an asymptotically
vanishing risk when, for some positive constant $c$,
\begin{align}\label{eq:lsst}
\rho k \geq
c\, \max\left(
\sqrt{\frac{\log |\CC|}{m}},
\frac{\log |\CC|}{m}
\right).
\end{align}
This condition was shown to be near-optimal in most regimes for the classes of $k$-sets and $k$-intervals, unless $k$ exceeds $\sqrt{n}$. In this latter and rather easier case, the simple non-localized squared sum statistic $T_s = \sum_{t=1}^m \left( \sum_{i=1}^n X_i^t\right)^2$ is near optimal. From \eqref{eq:lsst}, it is easy to see that the size of the class plays an important role, as a smaller class $\CC$ leads to a weaker sufficient condition for detection. In particular, the localized squared sum test has asymptotically vanishing risk  when
\begin{align*}
\textbf{k\text{-sets: }}&
\rho \geq c\, \max \left( \sqrt{\frac{\log n}{km}} , \frac{\log n}{m}
\right),
\\
\textbf{k\text{-intervals: }}&
\rho \geq c\,
\max
\left(\frac{1}{k} \sqrt{\frac{\log n}{m}} , \frac{\log n}{km}\right).
\end{align*}
Necessary conditions for detection  almost matching the previous sufficient conditions have been derived in \cite{arias2012detecting}. Although the dependence on the ambient dimension $n$ is only logarithmic, this can still be significant in regimes where $n$ is large but $m$ is small.

\subsection{Related Work}
A closely related problem is that of detecting non zero mean
components of a Gaussian vector $X$, referred to as the {\it
  detection-of-means} problem. This problem has received ample
attention in the literature, see, for instance,
\cite{ingster:97,bar02,DoJi04,ArCaHeZe08,ingster2009classification,AdBrDeLu10,hall2010innovated} and references therein. The
detection-of-means problem can be formulated as the multiple
hypothesis testing problem
\begin{align*}
H_0: \quad & X\sim \mathcal{N}(0, I_n),\\
H_1: \quad & X\sim \mathcal{N}(\mu \one_S, I_n),
\text{ for some } S \in \CC,
\end{align*}
where $\one_S$ is the indicator vector of $S$, $I_n$ is the identity matrix, and $\mu\neq 0$.
In other words, one needs to decide whether the components of $X$ are
independent standard normal random variables or they are independent 
normals with unit variance, and there is a (unknown) subset $S$ of $k$ components that have non-zero mean. The set of contaminated components $S$ is assumed to belong to a class $\CC$ of subsets of $[n]$. 
The behavior of the minimax risk has been analyzed for various class 
choices $\CC$ \cite{ingster:97,butucea2011detection,ArCaHeZe08,AdBrDeLu10}.
Detection and estimation in this model has been analyzed under
adaptive sensing in \cite{castro2012adaptive, haupt2009distilled},
where it is shown that, perhaps surprisingly, all sufficiently
symmetric classes $\CC$ lead to the same almost matching necessary and
sufficient conditions for detection.
This is quite different from the non-adaptive version of the problem
where size and structure of $\CC$ influence, in a significant way, 
possibilities of detection (see \cite{AdBrDeLu10}).

Recall that the correlation model of Section~\ref{sec:model}
can be rewritten as
\begin{align*}
H_0:\quad
&U^t_i = Y^t_i,\, i \in \{1, \ldots, n\},\\
H_1:\quad
&U^t_i =
\begin{cases}
Y^t_i,\, &i \notin S,\\
\sqrt{1-\rho} Y^t_i + \sqrt{\rho} N^t ,\, &i \in S
\end{cases}
\end{align*}
for some $S \in \CC$,
with $(Y^t_i), N^t$ independent standard normals, and that, as a consequence, the correlation model can be seen as a {\it random mean shift} model, with a slightly different normalization. However, most results on 
adaptive sensing for
detection-of-means heavily hinge on the independence assumption between coordinates, which is not applicable for the detection of correlations. In particular, we shall see that the picture is more subtle in the presence of correlations.

A second problem, perhaps even more related, is that of detection in sparse principal component analysis (sparse PCA) within the {\it rank one spiked covariance model}, defined  as the testing problem
\begin{align*}
H_0:\quad
&
X \sim \NN(0, I_n),\\
H_1:\quad
&
X \sim \NN(0, I_n + \theta u u^T),
\end{align*}
for some $u\in\R^n$ with $\|u\|_0 = k,
\,
\|u\|_2 =1$, 
where $\|u\|_0$ is the number of nonzero elements of $u$, and $\|u\|_2$ is the Euclidean norm of $u$. There is, also for this problem, a growing literature, see \cite{johnstone2009consistency, berthet2012optimal, cai2013optimal}. Note that when the coordinates of $u$ are constrained in $\{0, 1/\sqrt{k}\}$, we recover a problem akin to that of detection of positive correlations, but with {\it unnormalized variances} over the contaminated set.
The related problem of support estimation has been considered in \cite{amini2008high}
under the similar assumption that coordinates of $u$ are constrained in $\left\{0, \pm 1/\sqrt{k}\right\}$.

\subsection{Results and Contributions}

The main contribution of this paper is to show that adaptive sensing procedures can significantly outperform the best non-adaptive tests for the model in Section~\ref{sec:model}. We tackle the classes of 
$k$-intervals and $k$-sets. For $k$-intervals, necessary and sufficient conditions are almost matching. In particular, the number of measurements $m$ necessary and sufficient to ensure that the risk approaches zero has almost no dependence on the signal dimension $n$. This is in stark contrast with the non-adaptive sensing results, where it is necessary for $m$ to grow logarithmically with $n$.

For $k$-sets, we obtain sufficient conditions that still depend
logarithmically in $n$, but which improve nonetheless upon uniform
sensing in some regimes. Although not uniform, the proposed sensing
strategy is still non-adaptive. In addition to this, in a slightly
different model akin to that of sparse PCA mentioned above, we show
that all previous results (both non-adaptive and adaptive) carry on,
and we obtain a tighter sufficient condition for detection of
$k$-sets, that is nearly independent of the dimension $n$, and also
improves significantly over non-adaptive sensing. Our results are
summarized in Table \ref{tbl:summary}.
The paper is structured as follows. We
obtain a general lower bound in Section \ref{sec:lb}, and study
various classes of contaminated sets. In Section \ref{sec:tests}, we propose
procedures for $k$-sets and $k$-intervals. In Section
\ref{sec:unnormalized}, we prove a tighter sufficient condition under
a slightly different model, for $k$-sets. Finally, we conclude with a
discussion in Section \ref{sec:discussion}.

\begin{table*}[htdp]
\begin{center}
\begingroup
\tiny
\begin{tabular}{cc|c|c|c}
&&reference&$\rho k \to 0$&$\rho k \to \infty$\\
\hline
\multirow{5}{*}{$k$-sets}
&
necessary condition
&
Thm. \ref{thm:lb}
&
$\rho k \sqrt{m} \rightarrow \infty$
&
-
\\
&
sufficient condition
&
Prop. \ref{ub:ksets1}
&
$\rho \sqrt{km} \geq \sqrt{\log \frac{n}{k}},$ and
$\rho k m \geq \log \frac{n}{k}$
&
identical
\\
&
sufficient condition (unnormalized model)
&
Prop. \ref{ub:ksets2}
&
$\rho \sqrt{k m} \geq \log \log \frac{n}{k}$
&
identical
\\
&
sufficient condition (uniform, $k = o(\sqrt{n})$)
&
 \cite{arias2012detection}
&
$\rho \sqrt{km} \geq \sqrt{\log n},$ and
$\rho m \geq \log n$
&
identical
\\
&
necessary condition (uniform)
&
\cite{arias2012detection}
&
$\rho \sqrt{km} \geq \sqrt{\log\frac{n}{k^2}},$ and $\rho m \geq \log {n \over k^2}$
&
identical
\\
\hline
\multirow{4}{*}{$k$-intervals}
&
necessary condition
&
Thm. \ref{thm:lb}
&
$\rho k \sqrt{m} \rightarrow \infty$
&
-
\\
&
sufficient condition
&
Prop. \ref{ub:kintervals2}
&
$\rho k \sqrt{m} \geq \sqrt{\log \log \frac{n}{k}}$
&
$\rho k m \geq \log \log \frac{n}{k}$
\\
&
sufficient condition (uniform)
& \cite{arias2012detection}
&
$\rho k \sqrt{m} \geq \sqrt{\log {n \over k}}$
&
$\rho k m \geq \log {n\over k}$ 
\\
&
necessary condition (uniform)
& \cite{arias2012detection}
&
$\rho k \sqrt{m} \geq \sqrt{\log {n\over k}}$,
&
$\rho k m \geq \log {n\over k}$
\end{tabular}
\endgroup
\end{center}
\caption{Summary of results (constants omitted).}
\label{tbl:summary}
\end{table*}%

\subsection{Notation}

We denote by $\E_\P$ the expectation with respect to a distribution
$\P$. The Kullback-Leibler (KL) divergence between two probability
distributions $\P$ and $\Q$ such that $\P$ is absolutely continuous
with respect to $\Q$ is $\KL(\P \, ||\, \Q) =
\E_\P\left[\log\left({\mathrm{d}\P}/{\mathrm{d}\Q}\right)\right]$,
with $\mathrm{d}\P/\mathrm{d}\Q$ the Radon-Nikodym derivative of $\P$
with respect to $\Q$. When $\P$ and $\Q$ admit densities $f$ and $g$,
respectively, with respect to the same dominating measure, we write
$\KL(\P\,||\,\Q) = \KL(f\,||\,g)$. We denote by $\one_A$ the indicator
function of an event or condition $A$.

\section{Lower bounds}
\label{sec:lb}

We say that a sequence 
$z = (a^1, x^1, \ldots, a^M, x^M) \in \left( 2^{[n]}\times
\R^n
\right)^M $ is {\it $M$-admissible} if
$\sum_{t=1}^M |a^t| \leq M$.
Consider an adaptive testing procedure $(\AA, \phi)$, with query sequence $ (A^1, \ldots, A^M) \in \left(2^{[n]}\right)^M$, and
$(X^1, \ldots, X^M) \in \left(\R^n\right)^M$ the corresponding sequence of observations.
Let $S \in \CC \cup \{\emptyset\}$ be the set of contaminated coordinates.
For $t\in[M]$, we denote by $f_{A^t \,|\, H^{t-1}}(\cdot\, | \,h^{t-1})$ the probability mass function of $A^t$ given $H^{t-1} = h^{t-1}$, and  by $f_{X^t \,|\, A^t; \, S}(\cdot | a^t)$ the density of
 $X^t \,|\,  A^t = a^t$  over $\R^n$ with respect to a suitable dominating measure over $\R^n$ (e.g., the product of the Lebesgue measure and a point mass at $0$). Therefore, the joint sequence $Z = (A^1, X^1, \ldots, A^M, X^M)$  admits a density $f_S$ with respect to some appropriate dominating measure. For any $M$-admissible sequence $(a^1, x^1, \ldots, a^M, x^M)$, this density factorizes as
\begin{align*}
&f_S(a^1, x^1,  \ldots, a^M, x^M)
\\
&=
\prod_{t=1}^M
f_{A^t \,|\, H^{t-1}}(a^t \,|\, 
a^1, x^1, \ldots, a^{t-1}, x^{t-1}
)
\,
f_{X^t \,|\, A^t; \,S}(x^t \,|\, a^t).
\end{align*}
For concreteness, let the density $f_S$ be zero on any joint
subsequence that is not $M$-admissible. It is crucial to note that all
the terms in the factorization corresponding to the sensing strategy 
{(i.e., corresponding to the selection of $A^t$ given the history)}
do not depend on $S$. This is central to our arguments, as likelihood
ratios simplify. More precisely, for any $M$-admissible sequence
$(a^1, x^1, \ldots, a^M, x^M)$,
\begin{align*}
\frac{
f_\emptyset(a^1, x^1,  \ldots, a^M, x^M)
}
{
f_S(a^1, x^1,  \ldots, a^M, x^M)
}
&=
\prod_{t=1}^M
\frac{
f_{X^t \,|\, A^t;\,  \emptyset}(x^t \,|\, a^t)
}
{
f_{X^t \,|\, A^t;\,  S}(x^t \,|\, a^t)
}
\\
&=
\prod_{t=1}^M
\frac{
f_{X_{A^t}^t \,|\, A^t;\,  \emptyset}(x_{a^t}^t \,|\, a^t)
}
{
f_{X_{A^t}^t \,|\, A^t;\,  S}(x_{a^t}^t \,|\, a^t)
}
,
\end{align*}
where the second equality follows from the sensing model.

Likelihood ratios play a crucial role in the characterization of testing performance. In particular, a classical argument (see, e.g., \cite[Lemma 2.6]{tsybakov2009introduction}) shows that, for any distributions $\P,\Q$ over a common measurable space $\Omega$ and any measurable function $\phi : \Omega \rightarrow \{0, 1\}$,
\[
\P(\phi \neq 0) + \Q(\phi \neq 1) \geq \frac{1}{4} \exp\left(-\KL(\P\,||\,\Q)\right).
\]
Therefore
\begin{align*}
R^*
&= \inf_{(\AA, \phi)} \left[
\P_0^\AA(\phi \neq 0) + \max_{S \in \CC} \P_S^\AA(\phi \neq 1)
\right]
\\
&
= \inf_{(\AA, \phi)}  \max_{S \in \CC}
\left[
\P_0^\AA(\phi \neq 0) + \P_S^\AA(\phi \neq 1)
\right]
\\
&
\geq
 \inf_{\AA}  \max_{S \in \CC}
 \left[
\frac{1}{4}
\exp(-
\KL(\P_0^\AA \,||\, \P_S^\AA))
\right]
\\
&
=
\frac{1}{4}
\exp(-
\sup_\AA \min_{S \in \CC}
\KL(\P_0^\AA \,||\, \P_S^\AA)).
\end{align*}
This entails that  the minimax risk under adaptive sensing can be lower bounded by upper bounding the maximin KL divergence. 
Here,  in order to bound the maximum KL divergence, we will take an approach similar to \cite{castro2012adaptive} for detection-of-means under adaptive sensing, although our setup differs slightly. In  \cite{castro2012adaptive}, the testing procedures measure a single coordinate at a time, while we need multiple measures per step in order to capture correlations. We have the following necessary condition.
\begin{thm}\label{thm:lb}
Let $\CC$ be either the class of $k$-sets or $k$-intervals or disjoint 
$k$-intervals, and define
\[
D(\rho,k)=
 \min \left[\frac{\rho}{2(1-\rho)}, \rho^2(k+1)\right].
\]
Then the minimax risk $ R^*_\AS$ of adaptive testing procedures with a measurement budget of $M=mn$ coordinates is lower bounded as
\[
R^*_\AS \geq
\frac{\exp
\left(-
m k
D(\rho, k)
\right)}{4}
.
\]
{As a consequence, for the risk $R^*_\AS$ to converge to zero, it is necessary that $m k D(\rho, k) \to \infty$.}
\end{thm}
\begin{proof}
First remark the following:  for $\rho \leq 1/2$, and for any $A \subseteq [n]$,
\begin{align*}
\KL(\P_0|_{A}\,||\,
\P_S|_{A}
)
\leq
D(\rho,k)\,
 |A \cap S|.
 \end{align*}
The proof is given in Appendix  \ref{app:condition_cor}.
The KL divergence between the joint probability models can we written as
\begin{align*}
\KL
(\P_0^\AA \, |\, \P_S^\AA)
&=
\sum_{t=1}^M
\E_{\P_0^\AA} \left[
\E_{\P_0^\AA} \left[
\log \frac{f_{X_{A^t}^t | A^t;\, \emptyset}(x_{A^t}^t | A^t)}{f_{X_{A^t}^t | A_t;\, S}(x_{A^t}^t | A^t)}
\bigg| A^t
\right]
\right]
\\
&=\sum_{t=1}^M
\E_{\P_0^\AA} \left[
\KL(f_{X_{A^t}^t | A^t;\, \emptyset}
(
\cdot| A^t) \,||\,
f_{X_{A^t}^t | A^t;\, S}(\cdot| A^t)
)
\right]
\\
&=\sum_{t=1}^M
\E_{\P_0^\AA} \left[
\KL(\P_0|_{A^t}\,||\,
\P_S|_{A^t}
)
\right]
\\
&
\leq
D(\rho,k)
\sum_{t=1}^M
\E_{\P_0^\AA}
\left[
 |A^t \cap S|
\right]
\\
&
=
D(\rho, k)
\sum_{i \in S} b_i
\end{align*}
using the shorthand
$b_i = \sum_{t=1}^{M} \E_{\P_0^\AA}[\one_{i \in A^t}]$.
Hence,
\begin{align*}
\sup_{\mathcal{A}}
\min_{S \in \mathcal{C}}
\KL(\P_0^\AA \, ||\, \P_S^\AA)
&
\leq
D(\rho,k)
\,
\sup_{\mathcal{A}}
\,
\min_{S \in \mathcal{C}}
\,
\sum_{i \in S} b_i.
\end{align*}
Define the \emph{class complexity} 
\[
\cpl(\CC, M) = \sup_{\mathcal{A} \in \AS}
\left\{
\min_{S \in \CC}
\sum_{i \in S}
b_i
\,:\,
b \in \R_+^n,\\ \sum_{i=1}^n b_i \leq M
\right\}
.\]
For any sensing strategy $\AA$, it holds that
\[
\sum_{i=1}^n b_i
=
\sum_{t=1}^M
\E_{\P_0^\AA} [|A^t \cap S|]
\leq M,
\]
such that
\begin{align*}
\sup_{\mathcal{A}}
\min_{S \in \CC}
\KL(\P_0^\AA \, ||\, \P_S^\AA)
& \leq
D(\rho,k)
\,
\cpl(\CC, M)
.
\end{align*}
From \cite[Lemma~3.1]{castro2012adaptive}, we conclude that, for the both classes $\CC_k$ and $\DD_{[k]}$, respectively $k$-sets and disjoint $k$-intervals we have $\cpl(\CC_k, M) = \cpl(\DD_{[k]}, M) = \frac{Mk}{n} = mk$ (assuming without loss of generality for disjoint $k$-intervals that $n/k$ is an integer\footnote{If $n/k$ is not an integer, one can directly show that $\cpl(\DD_{[k]}, M)\leq 2mk$ and the result of the theorem for this class follows with $mk$ replaced by $2mk$.}). 
As $\cpl(\cdot, M)$ is decreasing with respect to set inclusion for any fixed $M$, $\cpl(\CC_{[k]},M)=mk$ as well, and the result follows.

\end{proof}

The lower bound argument in Theorem~\ref{thm:lb} yields the same lower
bound for detection using any of the three classes of interest. This
phenomenon is akin to what was observed in the context of
detection-of-means under adaptive sensing, where the lower bounds are the same provided the
classes of contaminated components are symmetric.
In this setting, it was shown in addition  
in \cite{castro2012adaptive} that
the condition in the lower bound is essentially sufficient
and therefore, unlike in the non-adaptive counterpart of the problem,
knowledge of the structure of $\CC$ does not make the
detection problem any easier. However, the problem of detection of correlations
considered here seems to be more subtle in that one lacks matching
upper bounds for all cases. Namely, we do not know whether:
(a) for detection-of-correlations
structure does not help; or (b) the lower bound is loose for some
classes, in particular the class of $k$-sets.

Recall that we are interested in the characterization of the regimes
for which the risk $R^*_\AS$ converges to zero as $m,k,n\rightarrow
\infty$. Clearly, if $\rho$ decays at a rate no faster than $1/k$,
the previous necessary condition for the risk to vanish asymptotically
is always satisfied. Nevertheless,
the lower bound gives an indication about the rate at which the risk
converges to zero. However, when $\rho = o\left(1/k\right)$ the
situation is different, and Theorem~\ref{thm:lb} leads to the
following necessary condition.
\begin{cor}%
\label{cor:lb}
Let $\CC$ denote either the class of $k$-sets, $k$-intervals or disjoint $k$-intervals, and suppose $\rho = o\left(1/k\right)$. For $R^*_\AS$ to converge to zero it is necessary that $\rho k \sqrt{m} \rightarrow \infty$.
\end{cor}
\begin{proof}
From the previous results, it is necessary that
\[
mk \,
\min \left[\frac{\rho}{2(1-\rho)}, \rho^2(k+1)\right]
\]
goes to infinity for the risk to converge to zero. This quantity is  asymptotically equivalent to $m \rho^2 k^2$, and $m\rho^2 k^2\rightarrow\infty$ if and only if $\rho k \sqrt{m} \rightarrow\infty$.
\end{proof}
Recall that a sufficient condition for non-adaptive detection of $k$-intervals with the localized squared sum test is
\[
\rho k \sqrt{m} > c\, \sqrt{\log(n)}
\text{ and }
\rho k m > c\, \log(n).
\]
When $\rho=o(1/k)$ one has, asymptotically, $\rho k < 1$ and the first condition is stronger than the second. Non-adaptive detection with $k$-intervals is thus possible asymptotically for $\rho k \sqrt{m} > c\, \sqrt{\log(n)}$.
This corresponds to the condition of Corollary \ref{cor:lb} up to a logarithmic factor in $n$, which implies that  in the case of $k$-intervals, one can improve at most by a factor logarithmic in $n$ with adaptive sensing. This can be still quite significant, and we show in Section \ref{sec:tests} that this can indeed be achieved.

\section{Adaptive tests}\label{sec:tests}

\subsection{The Case of $k$-intervals} \label{sec:st_disjoint}
In this section, we study the case of the class $\CC_{[k]}$ of intervals of length $k$. It is sufficient to work with the class $\DD_{[k]}$ of disjoint intervals for the following reason: assume that one has a procedure
for detection of disjoint $k$-intervals.
Then, for detection of general $k$-intervals, {\it this procedure can be applied as if the objective was detection of disjoint $k/2$-intervals}. Indeed, if $S$ is any $k$-interval, there exist at most two sets in $\DD_{[k/2]}$ that intersect $S$, and at least one of them, say $S'$, has a full intersection with $S$, i.e., $|S \cap S'| = k/2$. 
As a consequence, under mild conditions on the procedure, this leads to a sufficient condition for detection of $k$-intervals identical up to constants to that associated with the original procedure for disjoint $k$-intervals.
{
Since up to two of the disjoint intervals can contain contaminated coordinates,
the theoretical analysis still has to be slightly amended, but these technical modifications are straightforward for the methods that we propose.
To keep the presentation simple, we only show how to perform detection in the case of disjoint $k$-intervals.
}
{
Recall that $\DD_{[k]} = \{I_1, \ldots, I_{\lfloor n/k \rfloor}\}$, where
$I_j = \{(j - 1) k + 1, \ldots, j k \}$ for $j \in [\lfloor n/k \rfloor]$.
For simplicity, we assume that $n/k$ is an integer. As the intervals
are disjoint, the problem is equivalent to $n/k$ independent
hypothesis testing problems, each of them over vectors in $\R^k$ that
are mutually independent. Formally, this can be cast as a testing
problem over a matrix $Z \in \R^{{n \over k} \times k}$, where $Z$ has
independent standard Gaussian entries except under the alternative
where $Z$ has a single row whose entries are mutually correlated
standard Gaussian random variables with correlation $\rho$. In this framework, each row corresponds to one of the $n/k$ disjoint $k$-intervals.
}

In the context of support recovery from signals with independent
entries using adaptive sensing, \cite{malloy2011sequential, malloy2011limits} have proposed the sequential thresholding (ST) procedure, which is based on an
intuitive bisection idea. Although initially introduced for support
estimation, ST can be easily adapted to detection, and we present such
results here. In addition, we present a slight generalization to
signals with independent {\it vector entries}, which will allow us to apply the modified procedure to the disjoint $k$-intervals problem. 
We will also use the original ST procedure in Section \ref{sec:STmodified}, and for this reason, we first present the method using general notations here.
Let $\Q_0$ and $\Q_1$ be two probability distributions over $\R^{\tilde d}$, and let $Z \in \R^{\tilde{n} \times \tilde d}$ be a random matrix.  Consider the multiple testing problem defined as follows. Under the null, $Z$ has rows identically distributed according to $\Q_0$. Under the alternative, a small unknown subset of
$\tilde{k}$ rows of $Z$ are distributed according to $\Q_1$, while the
remaining rows are distributed according to $\Q_0$. In both cases, all
rows are independent.
More formally,  denote by $Z_1, \ldots, Z_{\tilde n}$ the rows of $Z$, such that the testing problem is 
\begin{align*}
H_0:\, & Z \sim \Q_0^{\, \otimes \tilde{n}},\\
H_1:\, & Z_i \sim \Q_0 \text{ for } i \notin S, \quad Z_i \sim \Q_1 \text{ for } i \in S,
\end{align*}
for some $S \in \CC$ with $|S|=\tilde{k}$,
where, as already mentioned, all rows are independent in both cases.  We refer to this testing problem as that of detection from {\it signals with independent (vector) entries}. The framework of adaptive sensing introduced in Section \ref{sec:sensingmodel} can be easily adapted to this model. In this case, in order to allow for vector entries, we consider that the experimenter is allowed to obtain samples from rows of $Z$, and that he can select which rows to query in a sequential manner as previously, under the  constraint that the total number of rows measured be less than $M$. 
We also refer to this straightforward extension as adaptive sensing, and we say that $\tilde m = M / \tilde n$ is the number of {\it measurements} (i.e., $\tilde m$ is the equivalent number of times the full matrix $Z$ was observed).

Sequential thresholding is a procedure for testing with adaptive sensing within the type of model  just mentioned.
Assume that $\Q_0$ and $\Q_1$ admit densities $f_0$ and $f_1$, respectively, with respect to some common dominating measure, and  for $i \in [n]$, denote by
\[
LR(f_1 | f_0; z_i^1, \ldots, z_i^{\tilde m})
=
\frac{\prod_{t=1}^{\tilde m} f_0(z_i^t)}{\prod_{t=1}^{\tilde m} f_1(z_i^t)}
\]
the likelihood ratio associated to i.i.d. observations $z_i^1, \ldots, z_i^{\tilde m} \in \R^{\tilde  d}$ of $Z_i$, the $i$-th row of $Z$.
ST proceeds as outlined in Figure \ref{alg:st}.
Initially, ST measures all $\tilde{n}$ rows $\bar m = \tilde{m}/4$ times, and throws away a fraction (of about half under the null) of the $\tilde{n}$ rows based on the values of the likelihood ratios. This is repeated with the remaining rows a number of times logarithmic in $\tilde{n}$, at which point ST calls detection if some coordinates have not been thrown away. This is illustrated in Figure \ref{fig:hist_st}.

\begin{figure}[h]
\centering
\includegraphics[width=7cm]{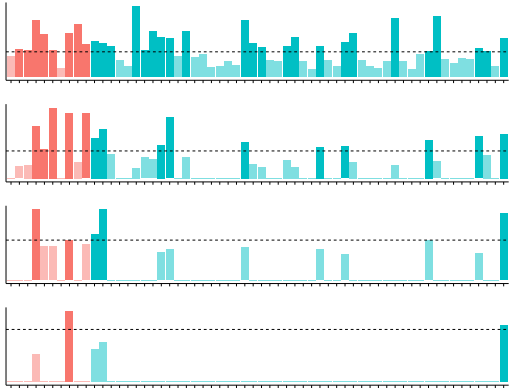}
\caption{Illustration of sequential thresholding with $k=10,\,  n=60$: 
contaminated coordinates 
are the first ten on the left. 
Bars depict likelihood ratios associated with each coordinate: at each step, coordinates with likelihood ratio below a threshold are thrown away. First step shown in top row, last step shown in bottom row.}
\label{fig:hist_st}
\end{figure}

\begin{figure}[h]
\centering
\fbox{
\parbox{0.9 \textwidth}{
\begin{algorithmic}
   \STATE Input: $K = \left\lfloor\log_2(\tilde{n}) \right\rfloor$ (number of steps),
   \STATE  $\quad\quad\quad \bar m = \frac{\tilde m}{4}$,
   \STATE $\quad\quad\quad  \gamma = \text{median}_{z_1^1, \ldots, z_1^{\bar m} \sim f_0}(LR(f_1 | f_0; z_1^1, \ldots, z_1^{\bar{m}}))$ (threshold)
   \footnotemark
   \STATE Initialization: $\mathcal{S}_0 = \{1, \ldots, \tilde{n}\}$
\FORALL{$r=1,\ldots,K$}
   \FORALL{$i \in \mathcal{S}_{r-1}$}
   \STATE measure $z_i^1, \ldots, z_i^{\bar{m}} \sim Z_i$
   \STATE compute $LR_i = LR(f_1 | f_0; z_i^1, \ldots, z_i^{\bar{m}})$
   \ENDFOR
   \STATE $\mathcal{S}_r = \{ i \in \mathcal{S}_{r-1} \,:\, LR_i > \gamma\}$
   \IF{$\sum_{r = 0}^K |\mathcal{S}_r| > \tilde n$}
   \STATE {\bfseries return}  no detection
   \ENDIF
   \ENDFOR
   \STATE {\bfseries return} detection if $\mathcal{S}_K \neq \emptyset$
\end{algorithmic}
}
}
\caption{Sequential thresholding procedure.}
\label{alg:st}
\end{figure}
\footnotetext{Here, $z_1^1, \ldots, z_1^{\bar m}$ denote without loss of generality observations of the first row, as rows are exchangeable under the null.}
The following result is easily deduced from the analysis of ST for support estimation.

\begin{prop}[Sufficient condition for ST]\label{prop:ST}
Assume $\tilde k / \tilde n \rightarrow 0$, and
\[
\liminf_{\tilde n \rightarrow \infty} \frac{\tilde{m}\KL(f_0 \,||\, f_1)}{
{4}
\log \log_2 \tilde{n}
} > 1,
\]
then the sequential thresholding procedure with a budget of $\tilde{m}$ measurements has risk tending to zero as $\tilde n$ goes to infinity.
\end{prop}
\begin{proof}
We begin by showing that the event of termination upon $\sum_{r = 0}^K
|\mathcal{S}_r| >\tilde n$ has an asymptotically vanishing
probability. Assume the alternative hypothesis with contaminated set
$S$. Then, similarly as in \cite[Proposition 4.1]{castro2012adaptive},
using Bernstein's inequality for sums of truncated hypergeometric
variables,
\[
P\left(
\sum_{r = 0}^K |\mathcal{S}_r| >\tilde  n
\right)
\leq \exp
\left(
-
\frac{\tilde n / 4 - \tilde k}{4 + \frac{2K}{3}}
\right),
\]
which converges to zero. The application of the Chernoff-Stein lemma as in \cite{malloy2011limits} allows us to bound the probability of error as follows.
The type I error of the procedure is 
bounded by
\[
\frac{\tilde n -\tilde k}{2^{K}}
.
\]
Let $E_{i,t}$ denote the event that the likelihood ratio is below $\gamma$ for coordinate $i$ at step $t$ (in which case, coordinate $i$ will not be included in $\mathcal{S}_t$). Without loss of generality, assume that $ 1 \in S$. 
The type II error is
\[
\Q_1
\left(
\cap_{i \in S}
\left(
\cup_{t=1}^K
E_{i,t}
\right)
\right)
\leq
\left(
K
\Q_1 \left(
E_{1, 1}
\right)
\right)^{\tilde k}.
\] 
{
We write $a \doteq e^{-\bar mD}$ for $\lim_{\bar m \rightarrow \infty} \frac{\log a}{\bar m} = D$.
}
From the Chernoff-Stein lemma,
\[
\Q_1 \left(
E_{1, 1}
\right)
\doteq
e^{-\bar  m\KL(f_0\, ||\, f_1)}.
\]
Hence, for $K = (1+\varepsilon_1) \log_2 n$ and $\varepsilon_2 > 0$, there exists $\bar m_0$ such that for $\bar m \geq \bar m_0$, the type II error is bounded by 
\begin{align*}
&\left(
K
e^{-\bar m(\KL(f_0\, ||\, f_1) - \varepsilon_2)}
\right)^{\tilde k}
\\
&\quad =
\exp
\left(\tilde k \log\left[(1 + \varepsilon_1) \log_2n\right] - \bar m \tilde k (\KL(f_0\, ||\, f_1) - \varepsilon_2)
\right).
\end{align*}
Hence, the risk goes to zero if for some $\varepsilon_1, \varepsilon_2 > 0$, it holds that
\[
\liminf_{\tilde n \rightarrow \infty} \frac{\bar{m}(\KL(f_0 \,||\, f_1) - \varepsilon_2)}{ \log\left[(1 + \varepsilon_1) \log_2n\right]} > 1.
\]
As a consequence, for the risk to go to zero, it is sufficient that 
\[
\liminf_{\tilde n \rightarrow \infty} \frac{\bar{m}\KL(f_0 \,||\, f_1)}{ \log \log_2n } > 1.
\]
{The result follows by substituting $\bar m$ with $\frac{\tilde m}{4}$.}
\end{proof}
Note that the ST procedure does not require knowledge of $\tilde{k}$. ST can be applied to the case of $k$-intervals, as we demonstrate in the next section.

We now show how the previous procedure can be used for adaptive detection with disjoint $k$-intervals. As before, we assume that $n/k$ is an integer. 
\begin{figure}[h]
\centering
\includegraphics[width=5cm]{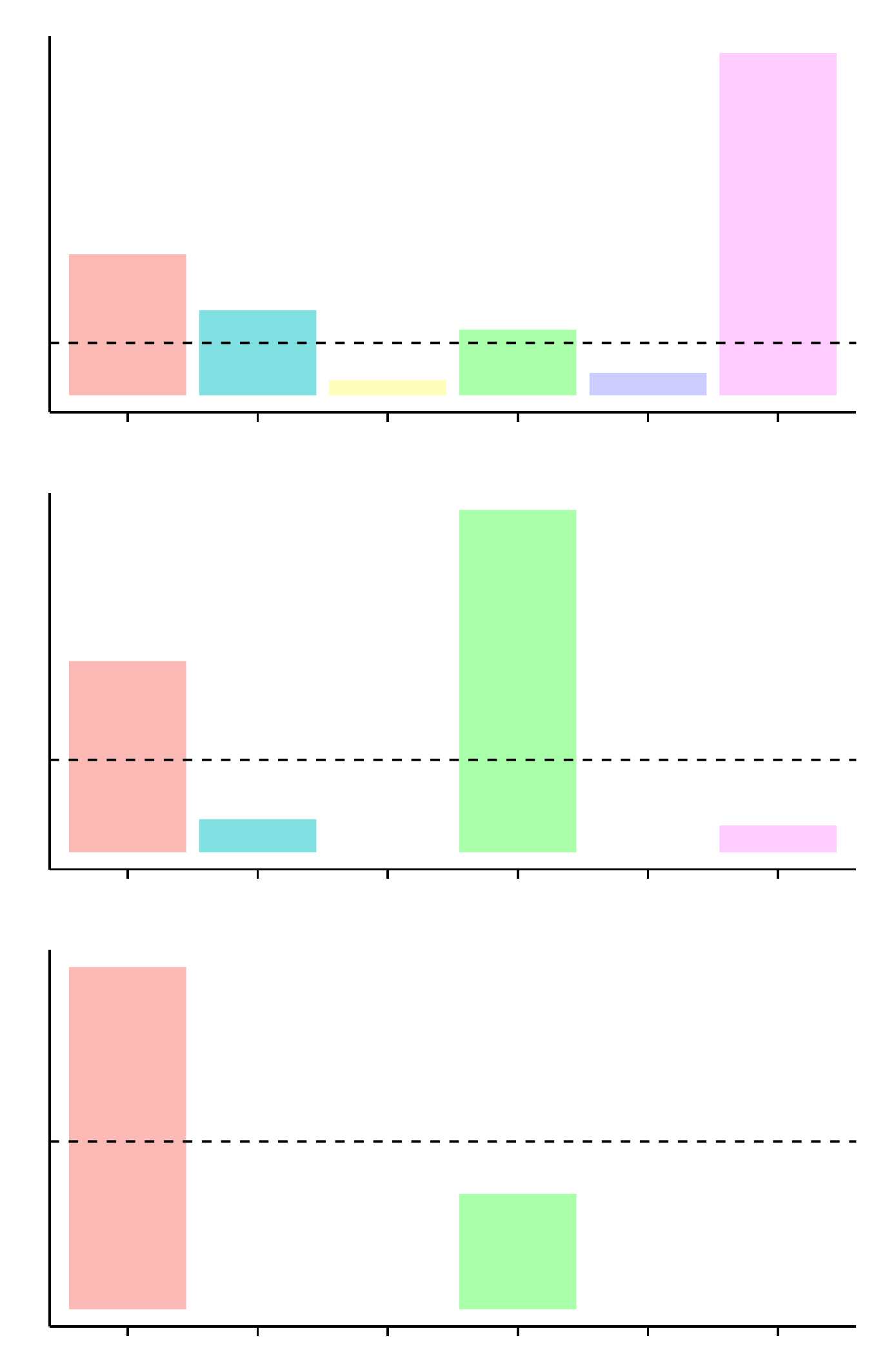}
\caption{Illustration of sequential thresholding for $k$-intervals, with $n/k = 6$ intervals of size $k$. Bars depict likelihood ratios associated with the intervals.}
\label{fig:hist_kintv}
\end{figure}
Define $\tilde n = n/k, \,\tilde k = 1,\, \tilde m = m,$ and $\tilde d = k$. 
Let $\Q_0 = \P_0|_{I_1}$ be the joint probability distribution over an interval under the null, and $\Q_1 = \P_{S}|_{S}$ be the joint probability distribution over the contaminated interval under the alternative with contaminated interval $S \in \DD_{[k]}$. Here, the choice of the interval used in $\Q_0$ does not matter, as intervals are exchangeable under the null hypothesis. We refer to the corresponding sequential thresholding procedure as {\it ST for disjoint $k$-intervals}. This procedure is illustrated in Figure \ref{fig:hist_kintv}.
This provides the following sufficient condition for detection of disjoint $k$-intervals.
\newcommand{\propkintervals}{
Assume that $\rho$ converges to zero.
{There exists numerical constants $C_3$ and $C_4$ such that,} when either 
\[
\rho k  \to \infty \quad \text{ and } \quad m \log(1+\rho k) \geq C_3 \log \log (n/k),
\]
or
\[
\rho k \to 0 \quad \text{ and }  \quad
\rho k \sqrt{m} \geq C_4 \sqrt{\log \log (n/k)},
\]
the sequential thresholding procedure for disjoint $k$-intervals has risk converging to zero.
}
\begin{prop}\label{ub:kintervals1}
\propkintervals
\end{prop}
\begin{proof}
 The detailed computations can be found in Appendix \ref{app:ub_kintervals1}.
Assume that $\rho k > 1$, then
\begin{align*}
\KL(\Q_0\,||\, \Q_1)
\geq
\frac{\log(1+\rho k)}{10}
.
\end{align*}
Similarly, when $\rho k < 1/2$ and $k > 32$,
\begin{align*}
\KL(\Q_0\,||\, \Q_1)
\geq
\frac{\rho^2 k^2}{16}
.
\end{align*}
Combined with Proposition \ref{prop:ST}, this gives the desired result.
\end{proof}

Consider the case where $\rho k  \to \infty$.  In that case, omitting constant factors, sequential thresholding would succeed for $m \geq \frac{\log \log
  (n)}{\log(1+ \rho k)}$.  Recall that uniform non-adaptive testing is
possible for $m \geq \frac{c\, \log n}{\rho k}$.  When $\rho k >
\log(n)$ asymptotically, both conditions are trivially satisfied for $m$ constant,
while when $\rho k < \log(n)$, we already improve upon non-adaptive
tests.  In spite of this, the dependence on $\rho k$ of our sufficient
condition when $\rho k  \to \infty$ is logarithmic, while it is only
linear for $\rho k \to 0$. This may appear surprising, as one may
argue the former case corresponds to a regime where the signal is stronger (and so the problem should be easier). 
{
However, this surprising fact is solely an artifact from the sequential thresholding procedure, and from the fact that  ST does not require knowledge of $k$. This results in a sufficient condition that is independent of $k$. In particular, it does not become easier to satisfy as $k$ increases, but it can be fixed through a small modification of the sensing methodology that we present in the following. 
}

In order to recover the same linear dependence in both cases, we
propose to add a subsampling stage prior to sequential
thresholding. This subsampling can be decided before any data is collected, and thus can be viewed as a non-adaptive aspect of the entire procedure.
Consider the simple deterministic subsampling scheme wherein one keeps the first $p$ coordinates per interval, for some $p \in \{2, \ldots, k\}$, and measures each $p$-tuple $\left\lfloor \frac{mn}{pn/k}\right\rfloor = \mkp$ times.
This  prompts the following question: is there a value of $p$ that allows one to detect more easily?
Define the $p$-truncated intervals as $I_j^p = \{(j-1) k+1, \ldots, (j-1) k + p\}$ for $j \in [n/k]$.
Formally, we consider the deterministic sensing strategy $\AA_p = (A^t)$ where for $t\in\left[\mkp\right]$,
\[
A^t = \bigcup_{j \in [n/k]} I_j^p.
\]
As this involves one simple testing problem per interval, the difficulty of testing is essentially characterized by the KL divergence $\KL(\P_0^{\AA_p}\,||\, \P_S^{\AA_p})$ between the distributions under the null and the alternative.
In this section, we make explicit the dependence of $\P_S$ on $p$ by using the notation $\P_S^p$.
Consider any fixed $S \in \DD_{[k]}$, then the best KL divergence that can be obtained is
\begin{align*}
\max_{p \in \{2, \ldots, k\}}
\KL\left(\P_0^{\AA_p}\,||\, \P_S^{\AA_p}\right)
&=
\max_{p \in \{2, \ldots, k\}}
\sum_{t=1}^{\mkp}
\KL(
\P_0^p
\,||\,
\P_S^p
)
\\
&=
\mkp \max_{p \in \{2, \ldots, k\}} \KL(
\P_0^p
\,||\,
\P_S^p
)
,
\end{align*}
which is independent of $S$
.
Due to nonlinearity in the KL divergence the optimal value of $p$ is generally different than $k$, as illustrated in Figure \ref{fig:p_star}.
\begin{figure}[h]
\centering
\includegraphics[width=6cm]{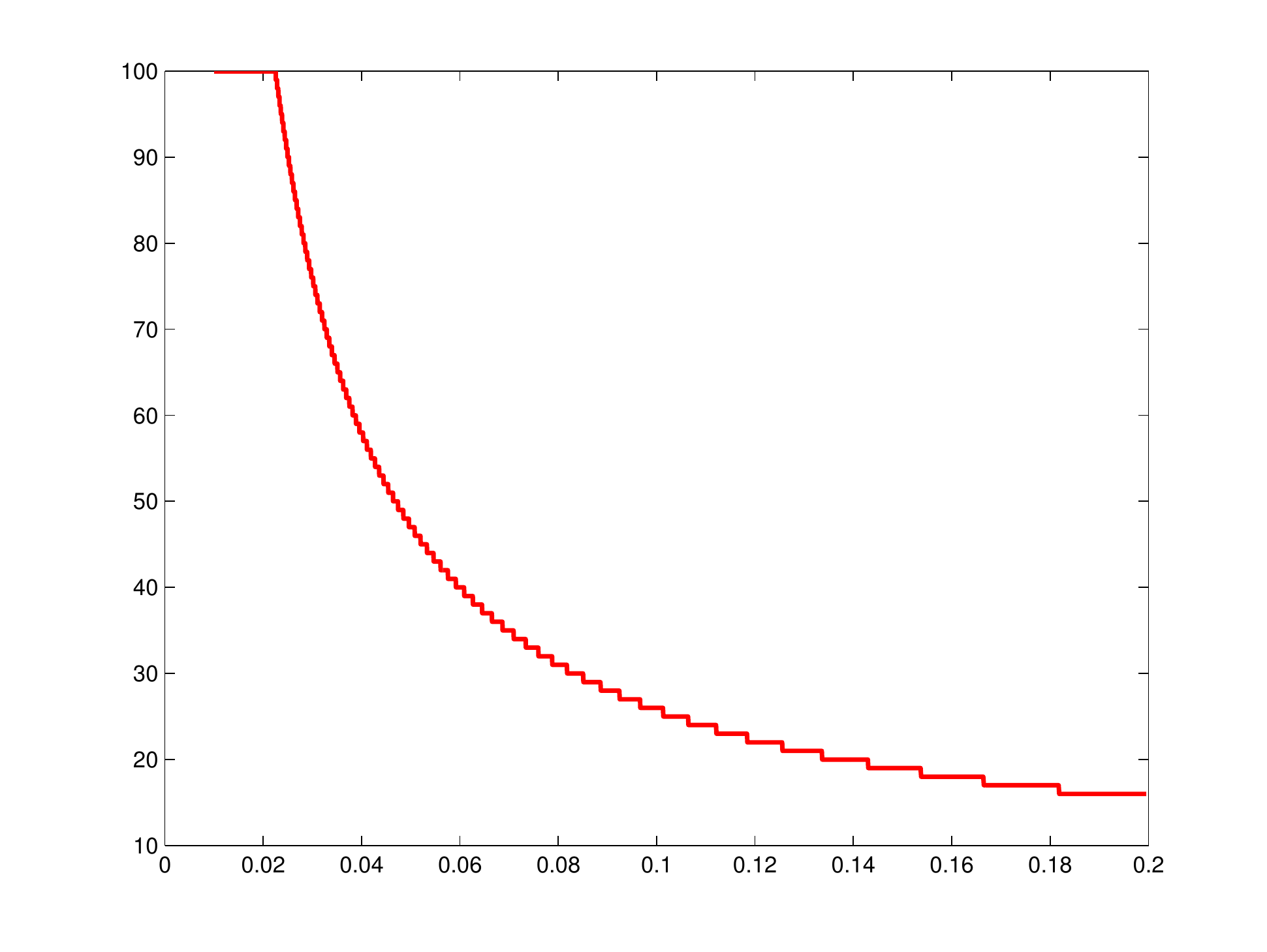}
\caption{Optimal $p$ as a function of $\rho$, for $k=100$.}
\label{fig:p_star}
\end{figure}
The optimal $p$ and corresponding optimal value seem hard to compute analytically, but numerical evidence shows that, for $\rho$ away from zero, the optimal $p$ is of the order of $\rho^{-1}$. This observation is sufficient for our purposes, and is formalized below.
Remark that when $\rho k < 1$, the optimal value of $p$ is clamped to $k$. 

Equipped with this subsampling stage when $\rho k \to \infty$, we can now modify the ST for $k$-intervals procedure as follows:
when $\rho k  \to \infty$, set $\tilde m = \mkp$, $\tilde d =  \left\lceil \frac{1}{\rho}\right\rceil$, and use only observations corresponding to $\tilde d$ coordinates per interval. We refer to this new procedure as the {\it modified sequential thresholding for disjoint $k$-intervals}.
\newcommand{\propkintervalstwo}{
Assume that $\rho$ converges to zero.  
{There exists numerical constants $C_5$ and $C_6$ such that,} when either
\[
 \rho k \to \infty \quad \text{ and }  \quad \rho k m \geq C_5 \log \log (n/k),
\]
or
\[
\rho k\to 0  \quad\text{ and } \quad
\rho k  \sqrt{m} \geq C_6 \sqrt{ \log \log (n/k)},
\]
the modified sequential thresholding procedure for disjoint $k$-intervals has risk converging to zero.
}
\begin{prop}
\label{ub:kintervals2}
\propkintervalstwo
\end{prop}
\begin{proof}
We have the following straightforward new lower bound: with $p=\left\lceil \frac{1}{\rho} \right\rceil$,
when $\rho k > 1$, we have $\left\lceil \frac{1}{\rho} \right\rceil < k + 1$, and as a consequence,
\[ 
\KL(
\P_0^p
\,||\,
\P_S^p
)
\geq
\frac{\log 2 -1/2}{2}
\geq
\frac{1}{11}
.\]
Although the lower bound appears weaker than previously, this corresponds to a setting where more measurements can be carried out.
The sufficient condition for ST leads to the result.
\end{proof}
The adaptive procedure allows us to obtain a mild dependence on the original dimension $n$ of the problem.
When $\rho = o(1/k)$, this sufficient condition almost matches the lower bound of Corollary \ref{cor:lb},
while when $\rho k \to \infty$, the sufficient condition is already satisfied for $m = \log \log (n/k)$.

\subsection{The Case of $k$-sets: Randomized Subsampling}
In this section, we consider the class $\CC_k$ of $k$-sets. In this case, we do not currently know whether a procedure along the lines of ST can be successfully applied. 
However, the idea of subsampling the coordinates can still be used to yield modest but important performance gains.
While for disjoint $k$-intervals a deterministic subsampling was sufficient, this is not the case for $k$-sets, where any deterministic subsampling that selects less than about $n - k$ coordinates cannot have risk converging to zero. For this reason, we consider a {\it randomized} subsampling of the coordinates.

Consider a sample $B$ of $\left\lfloor \frac{2 np}{k}\right\rfloor$ elements drawn without replacement from $[n]$ for some $p \geq 2$.
Let $\theta : \R^{\left\lfloor 2np/k \right\rfloor} \rightarrow \{0, 1\}$ be the localized squared sum test with ambient dimension $\left\lfloor \frac{2np}{k}\right\rfloor$, and contaminated sets $\CC = \CC_{\left\lfloor p \right\rfloor}$ of size $\lfloor p \rfloor$, and consider the sensing strategy defined by
\[A^1 = \ldots =  A^{\left\lfloor \frac{mk}{2p} \right\rfloor} = B.\]
We refer to the adaptive sensing procedure $((A^t), \theta)$ as the {\it randomized testing procedure}.
Define $Y = |B \cap S|$ (resp. $Y = 0$) under the alternative with contaminated $S \in \CC_k$ (resp. under the null), which is the number of contaminated elements in the subsample. Clearly $Y$ is a hypergeometric random variable with expectation $\frac{k}{n} \left\lfloor \frac{2n}{k} \, p\right\rfloor \in [2p-k/n, 2p]$. In words, we consider a subsample of the coordinates, with about $2p$ contaminated coordinates (in expectation) under the alternative, and we apply the (non-adaptive) localized squared sum test.

Note that the procedure is strictly non-adaptive, as the subsampling can be decided in advance. However, this sensing strategy is a bit different than uniform sensing, as not all coordinates are measured. Nonetheless, this allows one to detect under weaker conditions than with uniform non-adaptive sensing when $k$ is large enough.

\begin{prop}\label{ub:ksets1}
Let $2 \leq p \leq k$ such that $p$ goes to infinity. Assume that $\rho$ converges to zero and that
\begin{align*}
\rho m k \geq  \frac{C_1 \log \frac{2pn}{k}}{\left[1 - {1 \over m} - {1 \over k}\right]},
\quad
\text{and}
\quad
\rho \sqrt{mk}\geq \frac{C_1  \sqrt{\log \frac{2pn}{k}}}{\sqrt{1 - {1 \over m} - {1 \over k}}},
\end{align*}
for some constant $C_1$,
then the randomized testing procedure has risk converging to zero.
\end{prop}
\begin{proof}
Let $ \eta_I$ (resp. $\eta_{II}$) be the risk of type I (resp. of type II) for $\theta$.
The type I error of the randomized testing procedure is $p_I = \eta_I$.
Let $p_+ = P(Y \geq \lfloor p \rfloor)$ the probability of the sample containing at least $\lfloor p\rfloor$ contaminated elements, and $p_- = 1 - p_+$.
Note that since $\frac{2 n p}{k} \frac{k}{n} = 2p$ goes to infinity, we can assume that $Y$ is distributed according to a Poisson distribution with parameter $2p$, as this is asymptotically equivalent to the hypergeometric distribution.
Hence, we have
$
p_- = P(Y  < {\lfloor p \rfloor} ) \leq 
\left(
1+
\frac{p (2p)^p}{p!} 
\right) \exp(-2p).
$
Using $p! \geq \sqrt{2 \pi p} \left(\frac{p}{e}\right)^p$,  we have that $p_- \leq   \exp(-2p) + \sqrt{p} \exp(-p/4)$, which converges to zero.
The type II error of the randomized testing procedure is
$
p_{II} =
p_+ \eta_{II} + p_-(1 - \eta_I)
\leq
 \eta_{II} +  p_-.
$
It remains to show that $\eta_I$ and $\eta_{II}$ both go to zero. This follows from the sufficient conditions for the localized squared sum test, and from $\lfloor p \rfloor \left\lfloor \frac{mk}{2p}\right\rfloor \geq \frac{mk}{2} \left[ 1 - 1/p + \frac{2(1-p)}{mk}\right] \geq \frac{mk}{2} \left[1 -1/p -1/m\right]$.
Hence,  the sufficient conditions for the localized squared sum test $\theta$ provides the result.
\end{proof}

In particular, for $p = \log \log n$, it is sufficient that, omitting constants,
\begin{align*}
\rho m k \geq \log \frac{n}{k},
\quad
\rho\sqrt{mk} \geq  \sqrt{\log \frac{n}{k}},
\end{align*}
to ensure the detection risk converges to zero. This does not match the adaptive lower bound, and the dependence on $n$ is still logarithmic.
However, this already improves upon the setting of uniform non-adaptive sensing when $k \geq \frac{m}{\log n}$.
Indeed, recall that using uniform sensing, the sufficient condition is
\[
 \rho m \geq \log n, \quad \rho \sqrt{mk} \geq \sqrt{\log n}.
\]
The first condition is insensitive to subsampling, due to the dependence in $mk$, and we do not improve with respect to it. The second condition, however, only depends on $m$, and does not get easier to satisfy when $k$ is large. Hence, our result shows that it is more efficient when $k$ is large enough to reduce to a problem with an almost constant contaminated set size, but with an increased budget of full vector measurements.

\section{Unnormalized correlation model} \label{sec:unnormalized}
\subsection{Model and Extensions of Previous Results}
An alternative choice to the previous correlation model is the following {\it unnormalized model} with covariance matrix
\[
(\bar \Sigma_S)_{i,j} =
\begin{cases}
1, & i = j,\, i \notin S,\\
1 + \rho, & i = j,\, i \in S,\\
\rho, & i \neq j,\,\text{and}\, i,j \in S,\\
0 & \text{otherwise.}
\end{cases}
\]
under the alternative with contaminated set $S\in\CC$. This model is a special case of the \textit{rank one spiked covariance} model introduced in \cite{johnstone2001distribution}.  Observe that this correlation model can also be rewritten as
\begin{align*}
H_0:\quad
&X^t_i = Y^t_i,\, i \in \{1, \ldots, n\},\\
H_1:\quad
&X^t_i =
\begin{cases}
Y^t_i,\, &i \notin S,\\
Y^t_i + \sqrt{\rho} N^t ,\, &i \in S
\end{cases}
\text{ for some } S \in \CC,
\end{align*}
with $(Y^t_i), N^t$ independent standard normals. This can thus be interpreted as a random additive noise model, as for the model of Section~\ref{sec:model}.
Observe that our original correlation detection model is obtained 
by normalizing each component such that the components have unit variance.
This is a minor difference that does not essentially
change the difficulty of detection in the non-adaptive setting (indeed all upper and lower bounds proved in \cite{arias2012detecting}
can be reproved for this model with minor modifications). Interestingly, however, under 
adaptive sensing the information
provided by the higher variance in the contaminated components
can be exploited to give a major improvement over the normalized
model. This may be done by applying the sequential thresholding
algorithm to the squares of the components as described below.

In the following, for any quantity $X$ relative to the \emph{normalized} model of Section~\ref{sec:model}, we denote by $\bar{X}$ the corresponding quantity related to the unnormalized model.
All of previous results can be shown to hold for this model as well. As already mentioned, this includes the necessary and sufficient conditions of \cite{arias2012detecting} (Proposition \ref{app:kl_unnormalized} in Appendix), but also the lower bound of Theorem \ref{thm:lb} (Proposition \ref{app:condition_cor_unnormalized} in Appendix),  and sufficient conditions for $k$-sets and  $k$-intervals of Propositions \ref{ub:ksets1} and  \ref{ub:kintervals2} (Proposition \ref{app:ub_kintervals2_unnormalized} in Appendix). In particular, the procedures associated to the sufficient conditions can be used with little modifications.

\subsection{The case of $k$-sets}
\label{sec:STmodified}

The procedure proposed below combines randomized subsampling with sequential thresholding, in order to capitalize on the unnormalized model.
Consider the second moments $Y_i = X_i^2$. Under the alternative with contaminated set $S\in\CC$, $Y_i$ is distributed as follows: (a) for $i\notin S$, $Y_i$ is distributed according to a chi-squared distribution with one degree of freedom (that we denote by $\chi_1^2$), (b) for $i \in S$, $Y_i$ is distributed as $(1+\rho)\,\chi_1^2$.
Note that under our sensing model, it is perfectly legitimate to sample $A_1 = \{1\}, \ldots, A_n = \{n\}$, and thus obtain independent samples of each of the coordinates of the random vector. In particular, this allows us to obtain independent samples from the coordinates of $Y$. As a consequence, we can directly apply ST to detect increased variance over a subset of the coordinates.

As already mentioned, ST does not require knowledge of $k$, which results in a sufficient condition that is independent of $k$. This condition can, however, be significantly weakened using the random subsampling used in last section. As in Proposition \ref {ub:ksets1}, this is due to the fact that by subsampling, one can increase the budget of full vector measurements, while the decrease in the contaminated set size does not impact the sufficient condition for detection.
This is summarized in the following result, which can be proved similarly as Proposition \ref{ub:ksets1}.

\begin{prop}[Sufficient condition for ST+randomized subsampling]\label{prop:ST3}
Assume $\tilde k / \tilde n \rightarrow 0$, and
\[
\liminf_{\tilde n \rightarrow \infty} \frac{\tilde{m} \tilde k \KL(f_0 \,||\, f_1)}{(\log \log_2 \tilde{n})^2} > 1,
\]
then the sequential thresholding procedure with randomized subsampling $(p = \log \log_2 \tilde n)$ and a budget of $4 \tilde{m}$ full vector measurements has risk tending to zero as $\tilde n$ goes to infinity.
\end{prop}

Let $\tilde n = n,\, \tilde k = k,$ and $\tilde m = m$. Let $\Q_0$ be the $\chi_1^2$ distribution, and $\Q_1$ be the $(1+\rho)\, \chi_1^2$ distribution, both with respect to Lebesgue's measure. We consider the associated sequential thresholding procedure (with randomized subsampling), with the previous modification of sampling independent single coordinates. We refer to this procedure as  {\it variance thresholding}.
This leads to the following sufficient condition for detection.
\begin{prop}\label{ub:ksets2}
Assume that $\rho$ converges to zero and that
\[
\rho \sqrt{km} \geq  C_2 \log \log_2 n
\]
for some constant $C_2$.
Then, the risk of the variance thresholding procedure converges to zero.
\end{prop}
\begin{proof}
Let $g$ be the density of a $\chi_1^2$-distributed random variable, such that the density of a $(1+\rho)\chi_1^2$-distributed random variable is given by $\frac{1}{1+\rho}g\left(\frac{\cdot}{1+\rho}\right)$. Then, using $g(x) \propto x^{-1/2} e^{-x/2}$,
\begin{align*}
\KL(\chi_1^2\,||\,(1+\rho) \chi_1^2)
&=
\int_\R \log\left(\frac{g(x)}{\frac{1}{1+\rho}g\left(\frac{x}{1+\rho}\right)}\right) g(x) dx
\\&
=
\log(1+\rho)
+
\int_\R \log\left(\frac{x^{-1/2} e^{-x/2}}{\left(\frac{x}{1+\rho}\right)^{-1/2} e^{\frac{-x}{2(1+\rho)}}}\right) g(x) dx
\\
&=
\log(1+\rho)
+
\int_\R \log\left(\frac{e^{\frac{-\rho x}{2(1+\rho)}}}{\left(1+\rho\right)^{1/2}}\right) g(x) dx
\\
&=
\frac{\log(1+\rho)}{2}
-
\frac{\rho}{2(1+\rho)} \int_\R x g(x) dx.
\end{align*}
As the expectation of a $\chi_1^2$-distributed random variable is one, this leads to
\[
\KL(\chi_1^2\,||\,(1+\rho) \chi_1^2)
=\frac{1}{2}\left[
\log(1+\rho) - \frac{\rho}{1+\rho}
\right]
=
\frac{\rho^2}{4} + o(\rho^2)
.
\]
Plugging this expression into the sufficient condition of Proposition \ref{prop:ST3} provides the result.
\end{proof}
Assume for the following discussion that $\rho k \to 0$.
The necessary condition that we have established previously is that $\rho k \sqrt{m}$ goes to infinity. Neglecting the double log factor, the sufficient condition that we have just obtained is that 
{
$\rho \sqrt{k m}$ 
}
goes to infinity, which is stronger. Hence, there is a gap between the sufficient and necessary condition.
In particular,  that $\rho k \sqrt{m}$ goes to infinity was shown to be near-sufficient for detection with $k$-intervals, and the gap that we observe for $k$-sets does not allow us to conclude as to whether structure helps for detection (as is the case under non-adaptive sensing).

Recall that the unnormalized model is  similar to that of detection in the problem of sparse PCA. The method of {\it diagonal thresholding} (also referred to as {\it Johnstone's diagonal method}) is a simple and tractable method for detection (and support estimation) in sparse PCA (with uniform non-adaptive sensing), which consists in testing based on the diagonal entries of empirical covariance matrix - that is, the empirical variances. Hence, it is similar to the method that we consider here, except that we estimate variances based on independent samples for each coordinate.
Note that this last point is essential to our method.
Indeed, consider the opposite case where we do not use independent samples for each coordinates.
For the sake of illustration, assume 
$\rho =1$, such that the contaminated components are exactly equal.
In this case, the probability of throwing away one component is equal to that of throwing away {\it all} contaminated components, and failure will occur with fixed non small probability due to the use of  dependent samples.

Finally, it is noteworthy that a na\"{i}ve implementation of the optimal test in the non-adaptive setting has complexity $O(n^k)$, while with adaptive sensing, we obtain a procedure that can be carried out in time and space linear in $n$, and still improves significantly with respect to the non-adaptive setting.

\section{Discussion}
\label{sec:discussion}
We showed that for $k$-intervals, adaptive sensing allows one to reduce the logarithmic dependence  in $n$  of sufficient conditions for non-adaptive detection to a mild $\log \log n$, and that this is near-optimal in a minimax sense.

For $k$-sets, the story is less complete.
The sufficient condition obtained in the unnormalized model is still stronger than the sufficient condition obtained for $k$-intervals, and does not match our common lower bounds, which leaves open the question of {\it whether structure helps under adaptive sensing for detection of correlations}? The analogous question for detection-of-means has a negative answer, meaning structure does not provide additional information for detection. However, for detection-of-correlations a definite answer is still elusive. Another open question is to what extent adaptive sensing allows one to overcome the exponential computational complexity barrier that one can encounter in the non-adaptive setting.

Aside from the normalized and unnormalized correlation models, other types of models can be considered. A more general version of our normalized model has been analyzed in \cite{arias2012detecting}, where the correlations need not be all the same, leading to results that involve the mean correlation coefficient
$\rho_{\text{avg}} = \left(\sum_{i, j \in S\,:\, i \neq j} (\Sigma_S)_{i,j}\right) /\,  k(k-1)$.
In addition, we assume in most procedures that $\rho$ and/or $k$ are known, and it would be of interest to have procedures that do not require such knowledge.

\section{Proofs and computations}\label{app:proofs}
\subsection{Inequalities and KL divergences}\label{app:misc}
In this section, we collect elementary inequalities that we  use repeatedly in the computations.
\begin{align}
\label{ineq:logoneplus}
\text{For } x > -1,
&
\quad
\log(1+x) \leq x,\\
\text{For } x > 0,
&
\quad
\label{ineq:chunkub2}
\log(1+x)  + \frac{1}{1+x} -1\leq x^2,\\
\text{For } 0 < x < 1/2,
&
\quad
\label{ineq:logoneminusrho2}
\log(1-x) + \frac{1}{1-x} -1\leq 2 x^2,\\
\text{For } x < 1,
&
\quad
\label{ineq:chunkub3}
- \log(1-x)  -  \frac{1}{1-x} + 1\leq x^2,\\
\text{For } x \in ]-1, 1],
&
\quad
\label{ineq:chunklb}
\log(1+x) + \frac{1}{1+x} - 1 \geq \frac{x^2}{8},\\
\text{For } x \geq 1,
&
\quad
\label{ineq:loglb}
\log(1+x) 
\\
&
\qquad +  \frac{1}{1+x} - 1 \geq \frac{\log(1+x)}{5}.
\end{align}
The following expression of the KL divergence is used throughout the paper.

\begin{prop}\label{prop:KLs}
We have
\begin{align}
\label{eq:KLnormalized}
\KL(\P_0 \,||\, \P_S)
&=
\frac{\one_{k\geq 2}}{2}
\bigg[
k \left(
-1 +  \frac{1}{1 - \rho}
+\log(1-\rho)
\right)
\\
&
\qquad
-\left(
\frac{1}{1 - \rho}
+ \log (1 - \rho)
\right)
\\\nonumber
&
\qquad
+
\left(
\frac{1}{1 + \rho (k-1)}
 + \log(1 + \rho (k-1))
\right)
\bigg].
\end{align}
\end{prop}

\begin{proof}
The KL divergence between $\P_0$ and $\P_S$ can be computed using the standard formula for KL divergence between two centered Gaussian vectors, with covariance matrices
\[
\Sigma_0
= I_{n},
\quad
\Sigma_1 = \Sigma_S.
\]
When $k < 2$, the divergence is zero, and we will thus assume $k \geq 2$. Up to a simultaneous permutation of rows and columns,
\[
 \Sigma_S =
\left[
\begin{array}{cc}
I_{n-k} & \\
& J_\rho(k)
\end{array}
\right]
\]
where $J_\rho(k) \in \R^{k \times k}$ has unit diagonal and coefficients equal to $\rho$ everywhere else.
$J_\rho(k)$ is a symmetric matrix, hence diagonalizable, and has eigenvalues $1 -\rho$  with multiplicity $k -1$ and $1 + (k -1) \rho$  with multiplicity one. As a consequence, we have, for $k \geq 2$,
\begin{align*}
\log \det  \Sigma_S &= (k - 1) \log (1 - \rho) + \log(1 +  \rho(k -1) )\\
\trace  \Sigma_S^{-1} &= (n - k) + \frac{k -1}{1 - \rho} + \frac{1}{1 + \rho (k -1)}.
\end{align*}
The KL divergence is thus
\begin{align*}
\KL(\P_0 \,||\, \P_S)
&=
\frac{1}{2}
\left[
\trace (\Sigma_1^{-1} \Sigma_0)
- n
- \log(\det \Sigma_0 / \det \Sigma_1)
\right]\\
&=
\frac{1}{2}
\bigg[
(n-k) + \frac{k -1}{1 - \rho} + \frac{1}{1 + \rho(k-1)}
- n
\\
&
\qquad
+ (k - 1) \log (1 - \rho)  + \log(1 + \rho(k-1))
\bigg]\\
&=
\frac{1}{2}
\bigg[
k\left(
-1 +  \frac{1}{1 - \rho}
+\log(1-\rho)
\right)
\\
&
\qquad
-\left(
\frac{1}{1 - \rho}
+ \log (1 - \rho)
\right)
\\
&
\qquad
+
\left(
\frac{1}{1 + \rho(k-1)}
 + \log(1 + \rho(k-1))
\right)
\bigg].
\end{align*}
\end{proof}

\subsection{Proof of bound on KL divergence}\label{app:condition_cor}
\begin{proof}
First note since the KL divergences are independent of $n$, it is sufficient to use the expressions of Proposition \ref{prop:KLs} with a contaminated set of size $s = |A \cap S|\leq k $.
As previously, we assume $s \geq 2$, as the result is trivial otherwise.
Consider the expression for the KL divergence given in \eqref{eq:KLnormalized}.
Using \eqref{ineq:logoneplus}, we obtain
\begin{align*}
\KL(\P_0|_A \,||\, \P_S|_A) 
&=\KL(\P_0 \,||\, \P_{S\cap A}) \\
&\leq
\frac{1}{2}
\bigg[
s \left(
-1 +  \frac{1}{1 - \rho}
+\log(1-\rho)
+ \rho
\right)
\\
&
\qquad
-
\left(
 \frac{1}{1 - \rho}
+ \log (1 - \rho)
\right)
 + \left(
 \frac{1}{1 +  \rho}
- \rho
\right)
\bigg]\\
&=
\frac{1}{2}
\bigg[
s\left(
\rho
+
\frac{\rho}{1 - \rho}
+\log(1-\rho)
\right)
\\
&
\qquad
+ \frac{-2\rho}{1 - \rho^2}
- \log (1 - \rho)
- \rho
\bigg]
\\
&
\leq
\frac{\rho s}{2(1-\rho)}.
\end{align*}
Using
\eqref{ineq:chunkub2} and
\eqref{ineq:chunkub3},
we obtain
\begin{align*}
\KL
(\P_0 \,||\, \P_S)
&\leq
\frac{1}{2}
\bigg[
(s-1)^2 \rho^2
+
2s \rho^2
+
\rho^2
\bigg]
\\
&
=
\frac{\rho^2}{2}
\bigg[
(s-1)^2
+
2s
+
1
\bigg]
\\
&
\leq
\frac{\rho^2 s (k+1)}{2}
.
\end{align*}
\end{proof}

\subsection{Proof of Proposition \ref{ub:kintervals1}}\label{app:ub_kintervals1}
\begin{proof}
We have $\KL(\Q_0\,||\, \Q_1) = k f(\rho) + h(\rho)$ with
\begin{align*}
f(\rho) &=
\frac{1}{2}
\left[
(1 - \rho )^{-1} + \log(1 - \rho) -1
\right]
,\\
h(\rho) &=
\frac{1}{2}
\bigg[
-\left(
\frac{1}{1 - \rho}
+ \log (1 - \rho)
\right)
\\
&
\qquad
+
\left(
\frac{1}{1 + (p -1) \rho}
 + \log(1 + (p -1) \rho)
\right)
\bigg].
\end{align*}
As previously, using \eqref{ineq:chunklb}, $f(\rho) \geq \frac{\rho^2}{16}.$
Assume that $\rho k < 1$ and $k > 7$, then using \eqref{ineq:logoneminusrho2} and  \eqref{ineq:chunklb},
\begin{align*}
\KL(\Q_0\,||\, \Q_1)
&
\geq
\frac{\rho^2 k}{16}
+
h(\rho)
\\
&\geq
 \frac{\rho^2 k}{16}
-
\frac{1}{2}
\left[
1 + 2\rho^2
\right]
+
\frac{1}{2}
\left[
1 + \frac{\rho^2 (k-1)^2}{8}
\right]\\
&=
\rho^2
\left[
\frac{k (k-1)^2}{16} - 1
\right]
\\
&
\geq
\frac{(\rho k)^2}{32}.
\end{align*}
Now assume that $\rho k > 1$, then for $k > 32$,
\begin{align*}
\KL(\Q_0\,||\, \Q_1)
&
\geq
\frac{\rho^2 k}{16}
-
\frac{1}{2}
\left[
1 + 2\rho^2
\right]
\\
&
\qquad
+
\frac{1}{2}
\left[
\frac{1}{1 + (k -1) \rho}
 + \log(1 + (k -1) \rho)
\right]
\\
&\geq
\rho^2 \left[\frac{k}{16} - 1\right]
\\
&
\qquad
+
\frac{1}{2}
\left[
\frac{1}{1 + (k -1) \rho}
 + \log(1 + (k -1) \rho)
-1
\right]
\\
&\geq
\frac{\rho^2 k}{32}
+ \frac{\log(1 + (k -1) \rho)
-1}{2}.
\end{align*}
\end{proof}

\section{Extensions to unnormalized model}
\subsection{Uniform (non-adaptive) lower bound for detection of positive correlations}\label{app:non-adaptive}
\begin{prop}
For any class $\CC$, any $\rho \in [0, 0.9)$, the minimum risk in the normalized model (resp. the unnormalized model) under uniform (non-adaptive) sensing is bounded as
\begin{align*}
R^* &\geq \frac{1}{2} - \frac{1}{4} \sqrt{E\left[\cosh^m\left(\frac{8 \rho Z}{1-\rho}\right)\right] - 1}\\
\bar{R}^*&\geq \frac{1}{2} - \frac{1}{4} \sqrt{E\left[\cosh^m\left(8 \rho Z\right)\right] - 1}
\end{align*}
where $Z$ is the size of the intersection of two elements of $\CC$ drawn independently and uniformly at random.
\end{prop}
\begin{proof}
This is essentially a reproduction of the proof of \cite{arias2012detecting}
with minor modifications. The details are omitted.
\end{proof}

\subsection{Uniform (non-adaptive)  upper bound for detection of positive correlations}\label{app:non-adaptive2}
Let $H(b) = b -1 - \log b$ for $b>1$.

\begin{prop}
Under uniform (non-adaptive) sensing, the localized square-sum test that rejects when
\[
Y_{\text scan} = \max_{S \in \CC} \sum_{t=1}^m \left(\sum_{i \in S} X_i^t\right)^2
\]
exceeds
\[
\frac{1}{2}\left(\rho k^2 m + H^{-1}(3 \log |\CC|/m) -1) km\right)
\]
is asymptotically powerful when
\begin{align*}
\rho k \geq c_1 \max\left(
\sqrt{\frac{\log |\CC|}{m}},
\frac{\log |\CC|}{m}
\right)
\end{align*}
both for the normalized and unnormalized models.
\end{prop}
\begin{proof}
This is proved in \cite{arias2012detecting} for the normalized model. In the case of the unnormalized model, the test statistic is distributed as $k \chi_m^2$ under the null, and as $(k(1+\rho) + \rho k (k-1))\chi_m^2$ under the alternative, which changes only mildly the proof with respect to the normalized model.
\end{proof}

\subsection{KL divergences}
\begin{prop}\label{app:kl_unnormalized}
We have
\begin{align}
\label{eq:KLunnormalized}
{\KL}
(\bar \P_0 \,||\, \bar \P_S)
&=
\frac{\one_{k\geq 2}}{2}
\left[
-1 + \frac{1}{1+\rho k} + \log(1+\rho k)
\right].
\end{align}
\end{prop}

\begin{proof}

The KL divergence between $\bar \P_0$ and $\bar \P_S$ can be computed using the standard formula for KL divergence between two centered Gaussian vectors, with covariances matrices
\[
\Sigma_0
= I_{n},
\quad
\Sigma_1 = \bar \Sigma_S .
\]
When $k = 0$, the divergence is zero, and we will thus assume $k \geq 1$.
Up to a simultaneous permutation of rows and columns,
\[
 \bar\Sigma_S=
\left[
\begin{array}{cc}
I_{n - k} & \\
& I_k + K_\rho(k)
\end{array}
\right]
\]
where $K_\rho(k) \in \R^{k \times k}$ has coefficients equal to $\rho$ everywhere. Like previously, $ I_k + K_\rho(k)$ is diagonalizable, and has eigenvalue $1$ with multiplicity $k-1$, and eigenvalue $ 1 + \rho k$ with multiplicity one. As a consequence, for $k \geq 1$, we have
\begin{align*}
\log \det  \bar \Sigma_S &= \log( 1 + \rho k)\\
\trace  \bar \Sigma_S^{-1} &= (n - 1) + \frac{1}{1+\rho k }.
\end{align*}
This leads to
\begin{align*}
\KL(\bar \P_0  \,||\, \bar \P_S)
&=
\frac{1}{2}
\left[
\trace (\Sigma_1^{-1} \Sigma_0)
- n
- \log(\det \Sigma_0 / \det \Sigma_1)
\right]\\
&=
\frac{1}{2}
\left[
(n - 1) - n + \frac{1}{1+\rho k} + \log(1+\rho k)
\right].
\end{align*}
\end{proof}

\begin{prop}\label{app:condition_cor_unnormalized}
For any $A \subset [n]$,
\[
\KL(\bar \P_0 |_A \,||\, \bar \P_S |_A) \leq \min\left[\frac{\rho}{2}, \frac{\rho^2 k}{2} \right] \, |A \cap S |.
\]
\end{prop}
\begin{proof}
First note since the KL divergences are independent of $n$, it is sufficient to use the expressions of Proposition \ref{prop:KLs} with a contaminated set of size $s =|A \cap S|$.
As previously, we assume $s \geq 1$, as the result is trivial otherwise.
Consider the unnormalized model, with KL divergence given in \eqref{eq:KLunnormalized}.
Using \eqref{ineq:logoneplus}, we obtain
\[
\KL(\bar \P_0|_A \,||\,  \bar \P_S|_A) 
=
\KL(\bar \P_0 \,||\,  \bar \P_{A \cap S}) \leq \frac{\rho s}{2}.
\]
Using \eqref{ineq:chunkub2} we obtain
\[
\KL(\bar \P_0|_A \,||\,  \bar \P_S|_A) 
=
\KL(\bar \P_0 \,||\, \bar  \P_{A \cap S}) \leq \frac{\rho^2 s^2}{2}
\leq  \frac{\rho^2 s k}{2}.
\]
Combining these last two inequalities yields the desired result.
\end{proof}

\begin{prop}\label{app:ub_kintervals1_unnormalized}
\propkintervals
\end{prop}
\begin{proof}
For the unnormalized model, when $\rho k > 1$, using \eqref{ineq:loglb},
\[
\KL(\bar \Q_0\,||\, \bar \Q_1)
\geq
\frac{\log(1+\rho k)}{10}.
\]
When $\rho k < 1$, using \eqref{ineq:chunklb},
\[
\KL(\bar \Q_0\,||\, \bar \Q_1)
\geq
\frac{(\rho k)^2}{16}.
\]
\end{proof}

\begin{prop}\label{app:ub_kintervals2_unnormalized}
\propkintervalstwo
\end{prop}
\begin{proof}
For the unnormalized model with $p=\left\lceil \frac{1}{\rho} \right\rceil$,
when $\rho k > 1$, we have $\left\lceil \frac{1}{\rho} \right\rceil < k + 1$, and as a consequence,
\[
\KL(\bar \P_0^p \,||\, \bar \P_S^p)
\geq
\frac{\log 2 -1/2}{2}
\geq
\frac{1}{11}
.\]
\end{proof}

% \bibliographystyle{plain}
% \bibliography{biblio}

\begin{thebibliography}{10}

\bibitem{AdBrDeLu10}
Louigi Addario-Berry, Nicolas Broutin, Luc Devroye, and G{\'a}bor Lugosi.
\newblock On combinatorial testing problems.
\newblock {\em The Annals of Statistics}, 38:3063--3092, 2010.

\bibitem{akyildiz2002wireless}
Ian~F Akyildiz, Weilian Su, Yogesh Sankarasubramaniam, and Erdal Cayirci.
\newblock Wireless sensor networks: a survey.
\newblock {\em Computer networks}, 38(4):393--422, 2002.

\bibitem{amini2008high}
Arash~A. Amini and Martin~J. Wainwright.
\newblock High-dimensional analysis of semidefinite relaxations for sparse
  principal components.
\newblock In {\em Information Theory, 2008. ISIT 2008. IEEE International
  Symposium on}, pages 2454--2458. IEEE, 2008.

\bibitem{arias2012detecting}
Ery Arias-Castro, S{\'e}bastien Bubeck, and G{\'a}bor Lugosi.
\newblock Detecting positive correlations in a multivariate sample.
\newblock {\em arXiv preprint arXiv:1202.5536}, 2012.

\bibitem{arias2012detection}
Ery Arias-Castro, S{\'e}bastien Bubeck, and G{\'a}bor Lugosi.
\newblock Detection of correlations.
\newblock {\em The Annals of Statistics}, 40(1):412--435, 2012.

\bibitem{arias2011fundamental}
Ery Arias-Castro, Emmanuel~J. Candes, and Mark~A. Davenport.
\newblock On the fundamental limits of adaptive sensing.
\newblock {\em Information Theory, IEEE Transactions on}, 59(1):472--481, 2013.

\bibitem{ArCaHeZe08}
Ery Arias-Castro, Emmanuel~J. Cand{\`e}s, Hannes Helgason, and Ofer Zeitouni.
\newblock Searching for a trail of evidence in a maze.
\newblock {\em The Annals of Statistics}, 36:1726--1757, 2008.

\bibitem{bar02}
Yannick Baraud.
\newblock Non-asymptotic minimax rates of testing in signal detection.
\newblock {\em Bernoulli}, 8:577--606, 2002.

\bibitem{barber1997detecting}
Brad~M Barber and John~D Lyon.
\newblock Detecting long-run abnormal stock returns: The empirical power and
  specification of test statistics.
\newblock {\em Journal of financial economics}, 43(3):341--372, 1997.

\bibitem{berthet2012optimal}
Quentin Berthet and Philippe Rigollet.
\newblock Optimal detection of sparse principal components in high dimension.
\newblock {\em Ann. Statist.}, 41(1):1780--1815, 2013.

\bibitem{butucea2011detection}
Cristina Butucea and Yuri~I. Ingster.
\newblock Detection of a sparse submatrix of a high-dimensional noisy matrix.
\newblock {\em arXiv preprint arXiv:1109.0898}, 2011.

\bibitem{cai2013optimal}
Tony Cai, Zongming Ma, and Yihong Wu.
\newblock Optimal estimation and rank detection for sparse spiked covariance
  matrices.
\newblock {\em arXiv preprint arXiv:1305.3235}, 2013.

\bibitem{castro2012adaptive}
Rui~M. Castro.
\newblock Adaptive sensing performance lower bounds for sparse signal
  estimation and testing.
\newblock {\em arXiv preprint arXiv:1206.0648}, 2012.

\bibitem{chandola2009anomaly}
Varun Chandola, Arindam Banerjee, and Vipin Kumar.
\newblock Anomaly detection: A survey.
\newblock {\em ACM Computing Surveys (CSUR)}, 41(3):15, 2009.

\bibitem{chennear}
Yuxin Chen and Andreas Krause.
\newblock Near-optimal batch mode active learning and adaptive submodular
  optimization.
\newblock In {\em ICML}, 2013.

\bibitem{DoJi04}
David Donoho and Jiashun Jin.
\newblock Higher criticism for detecting sparse heterogeneous mixtures.
\newblock {\em The Annals of Statistics}, 32:962--994, 2004.

\bibitem{hall2010innovated}
Peter Hall and Jiashun Jin.
\newblock Innovated higher criticism for detecting sparse signals in correlated
  noise.
\newblock {\em The Annals of Statistics}, 38(3):1686--1732, 2010.

\bibitem{haupt2012SCS}
Jarvis Haupt, Richard Baraniuk, Rui Castro, and Robert Nowak.
\newblock Sequentially designed compressed sensing.
\newblock In {\em Statistical Signal Processing Workshop (SSP), 2012 IEEE},
  pages 401--404. IEEE, 2012.

\bibitem{haupt2009distilled}
Jarvis Haupt, Rui Castro, and Robert Nowak.
\newblock Distilled sensing: Selective sampling for sparse signal recovery.
\newblock In {\em Proc. 12th International Conference on Artificial
  Intelligence and Statistics (AISTATS)}, pages 216--223, 2009.

\bibitem{hero2012hub}
Alfred Hero and Bala Rajaratnam.
\newblock Hub discovery in partial correlation graphs.
\newblock {\em Information Theory, IEEE Transactions on}, 58(9):6064--6078,
  2012.

\bibitem{hofmeyr1998intrusion}
Steven~A Hofmeyr, Stephanie Forrest, and Anil Somayaji.
\newblock Intrusion detection using sequences of system calls.
\newblock {\em Journal of computer security}, 6(3):151--180, 1998.

\bibitem{ingster:97}
Y.~Ingster.
\newblock Some problem of hypothesis testing leading to infinitely divisible
  distributions.
\newblock {\em Mathematical Methods of Statistics}, 6:47--69, 1997.

\bibitem{ingster2009classification}
Yuri~I Ingster, Christophe Pouet, and Alexandre~B Tsybakov.
\newblock Classification of sparse high-dimensional vectors.
\newblock {\em Philosophical Transactions of the Royal Society A: Mathematical,
  Physical and Engineering Sciences}, 367(1906):4427--4448, 2009.

\bibitem{janakiram2006outlier}
D~Janakiram, V~Adi Mallikarjuna~Reddy, and AVU Phani~Kumar.
\newblock Outlier detection in wireless sensor networks using bayesian belief
  networks.
\newblock In {\em Communication System Software and Middleware, 2006. Comsware
  2006. First International Conference on}, pages 1--6. IEEE, 2006.

\bibitem{johnstone2001distribution}
Iain~M. Johnstone.
\newblock On the distribution of the largest eigenvalue in principal components
  analysis.
\newblock {\em The Annals of statistics}, 29(2):295--327, 2001.

\bibitem{johnstone2009consistency}
Iain~M. Johnstone and Arthur~Yu Lu.
\newblock On consistency and sparsity for principal components analysis in high
  dimensions.
\newblock {\em Journal of the American Statistical Association}, 104(486),
  2009.

\bibitem{jung2002flash}
Jaeyeon Jung, Balachander Krishnamurthy, and Michael Rabinovich.
\newblock Flash crowds and denial of service attacks: Characterization and
  implications for cdns and web sites.
\newblock In {\em Proceedings of the 11th international conference on World
  Wide Web}, pages 293--304. ACM, 2002.

\bibitem{kulldorff2005space}
Martin Kulldorff, Richard Heffernan, Jessica Hartman, Renato Assun{\c{c}}ao,
  and Farzad Mostashari.
\newblock A space--time permutation scan statistic for disease outbreak
  detection.
\newblock {\em PLoS medicine}, 2(3):e59, 2005.

\bibitem{lin2005approximations}
Jessica Lin, Eamonn Keogh, Ada Fu, and Helga Van~Herle.
\newblock Approximations to magic: Finding unusual medical time series.
\newblock In {\em Computer-Based Medical Systems, 2005. Proceedings. 18th IEEE
  Symposium on}, pages 329--334. IEEE, 2005.

\bibitem{malloy2011limits}
Matt Malloy and Robert Nowak.
\newblock On the limits of sequential testing in high dimensions.
\newblock In {\em Signals, Systems and Computers (ASILOMAR), 2011 Conference
  Record of the Forty Fifth Asilomar Conference on}, pages 1245--1249. IEEE,
  2011.

\bibitem{malloy2011sequential}
Matthew Malloy and Robert Nowak.
\newblock Sequential analysis in high-dimensional multiple testing and sparse
  recovery.
\newblock In {\em Information Theory Proceedings (ISIT), 2011 IEEE
  International Symposium on}, pages 2661--2665. IEEE, 2011.

\bibitem{moore2006inferring}
David Moore, Colleen Shannon, Douglas~J Brown, Geoffrey~M Voelker, and Stefan
  Savage.
\newblock Inferring internet denial-of-service activity.
\newblock {\em ACM Transactions on Computer Systems (TOCS)}, 24(2):115--139,
  2006.

\bibitem{noble2003graph}
Caleb~C Noble and Diane~J Cook.
\newblock Graph-based anomaly detection.
\newblock In {\em Proceedings of the ninth ACM SIGKDD international conference
  on Knowledge discovery and data mining}, pages 631--636. ACM, 2003.

\bibitem{shekhar2001detecting}
Shashi Shekhar, Chang-Tien Lu, and Pusheng Zhang.
\newblock Detecting graph-based spatial outliers: algorithms and applications
  (a summary of results).
\newblock In {\em Proceedings of the seventh ACM SIGKDD international
  conference on Knowledge discovery and data mining}, pages 371--376. ACM,
  2001.

\bibitem{song2007conditional}
Xiuyao Song, Mingxi Wu, Christopher Jermaine, and Sanjay Ranka.
\newblock Conditional anomaly detection.
\newblock {\em Knowledge and Data Engineering, IEEE Transactions on},
  19(5):631--645, 2007.

\bibitem{srinivas2009Gaussian}
Niranjan Srinivas, Andreas Krause, Sham Kakade, and Matthias Seeger.
\newblock Gaussian process optimization in the bandit setting: No regret and
  experimental design.
\newblock In {\em Proc. International Conference on Machine Learning (ICML)},
  2010.

\bibitem{thottan2003anomaly}
Marina Thottan and Chuanyi Ji.
\newblock Anomaly detection in ip networks.
\newblock {\em Signal Processing, IEEE Transactions on}, 51(8):2191--2204,
  2003.

\bibitem{tsybakov2009introduction}
Alexandre~B. Tsybakov.
\newblock {\em Introduction to nonparametric estimation}.
\newblock Springer, 2009.

\bibitem{van2006anomaly}
Tran Van~Phuong, Le~Xuan Hung, Seong~Jin Cho, Young-Koo Lee, and Sungyoung Lee.
\newblock An anomaly detection algorithm for detecting attacks in wireless
  sensor networks.
\newblock In {\em Intelligence and Security Informatics}, pages 735--736.
  Springer, 2006.

\bibitem{wang2002detecting}
Haining Wang, Danlu Zhang, and Kang~G Shin.
\newblock Detecting syn flooding attacks.
\newblock In {\em INFOCOM 2002. Twenty-First Annual Joint Conference of the
  IEEE Computer and Communications Societies. Proceedings. IEEE}, volume~3,
  pages 1530--1539. IEEE, 2002.

\end{thebibliography}

\end{document}